\documentclass{amsart}
\usepackage[utf8]{inputenc}

\usepackage{graphics}
\usepackage{thmtools}
\usepackage{wasysym}
\usepackage[export]{adjustbox}
\usepackage[T1]{fontenc}    

\usepackage{amsthm}
\usepackage{amsbsy,amsmath,amssymb,amscd,amsfonts}
\usepackage[pagebackref=true]{hyperref}

\usepackage{graphicx,float,latexsym,color}
\usepackage[font={scriptsize,it}]{caption}
\usepackage{subcaption}

\usepackage{makecell}

\usepackage[dvipsnames]{xcolor}

\newtheorem*{theorem*}{Theorem}
\newtheorem{observation}{Observation}
\newtheorem{proposition}{Proposition}
\newtheorem{conjecture}{Conjecture}
\newtheorem{corollary}{Corollary}
\newtheorem{lemma}{Lemma}
\theoremstyle{remark}
\newtheorem{remark}{Remark}
\theoremstyle{definition}
\newtheorem{definition}{Definition}

\hypersetup{
    pdftoolbar=true,        
    pdfmenubar=true,        
    pdffitwindow=false,     
    pdfstartview={FitH},    
    colorlinks=true,       
    linkcolor=OliveGreen,          
    citecolor=blue,        
    filecolor=black,      
    urlcolor=red           
}

\usepackage[nameinlink,capitalize,noabbrev]{cleveref}

\usepackage{lineno}

\usepackage{comment}
\usepackage{microtype}

\graphicspath{{pics/}}
\newcommand{\T}{\mathcal{T}}
\newcommand{\Q}{\mathcal{Q}}

\newcommand{\C}{\mathcal{C}}

\newcommand{\B}{\mathcal{B}}

\renewcommand{\H}{\mathcal{H}}
\newcommand{\D}{\mathcal{D}}
\renewcommand{\P}{\mathcal{P}}
\newcommand{\R}{\mathcal{R}}

\newcommand{\torp}[2]{\texorpdfstring{#1}{#2}}

\title[Parabola-Inscribed Poncelet Polygons]{Parabola-Inscribed Poncelet Polygons\\Derived from the Bicentric Family}
\author{Filipe Bellio}
\thanks{F. Bellio, ENS-Lyon, France. \texttt{filipe.bellio-da-nobrega@ens-lyon.fr}}
\author{Ronaldo Garcia}
\thanks{R. Garcia, Inst. Mat. e Estatística, Univ. Fed. de Goiás, Goiânia, Brazil. \texttt{ragarcia@ufg.br}}
\author[D. Reznik]{Dan Reznik}
\thanks{D. Reznik$^*$, Data Science Consulting, Rio de Janeiro, Brazil. \texttt{dreznik@gmail.com}}
\date{}
\begin{document}

\maketitle

\vspace{-0.75cm}
\begin{abstract}
We study loci and properties of a Parabola-inscribed family of Poncelet polygons whose caustic is a focus-centered circle. This family is the polar image of a special case of the bicentric family with respect to its circumcircle. We describe closure conditions, curious loci, and new conserved quantities. 
\end{abstract}



\section{Introduction}
This is a continuation of our investigation of Euclidean phenomena of Poncelet families \cite{garcia2020-family-ties,garcia2021-impa,garcia2020-new-properties,reznik2020-ballet}. Recall Poncelet's porism: specially-chosen pairs of conics $\C,\C'$ admit a one-parameter family of polygons inscribed in $\C$ while simultaneously circumscribed about $\C'$ \cite{bos-1987,centina2016a,dragovic11}.

Here we consider a certain family such that $\C$ is a parabola $\P$ while $\C'$ is a circle centered on the focus of $\P$. As shown in \cref{fig:bic-tang-n3}, this is simply the polar image of the {\em bicentric family} (interscribed between two circles) with respect to its circumcircle, see \cref{app:vtx,app:bic} for construction details. We derive closure conditions for this new family for $N=3,4,5,6$ cases ($N$ is the number of sides) and describe some of its properties and loci of associated points. Also considered is its polar image with respect to $\P$.

\begin{figure}
    \centering
    \includegraphics[trim=180 0 0 80,clip,width=\textwidth,frame]{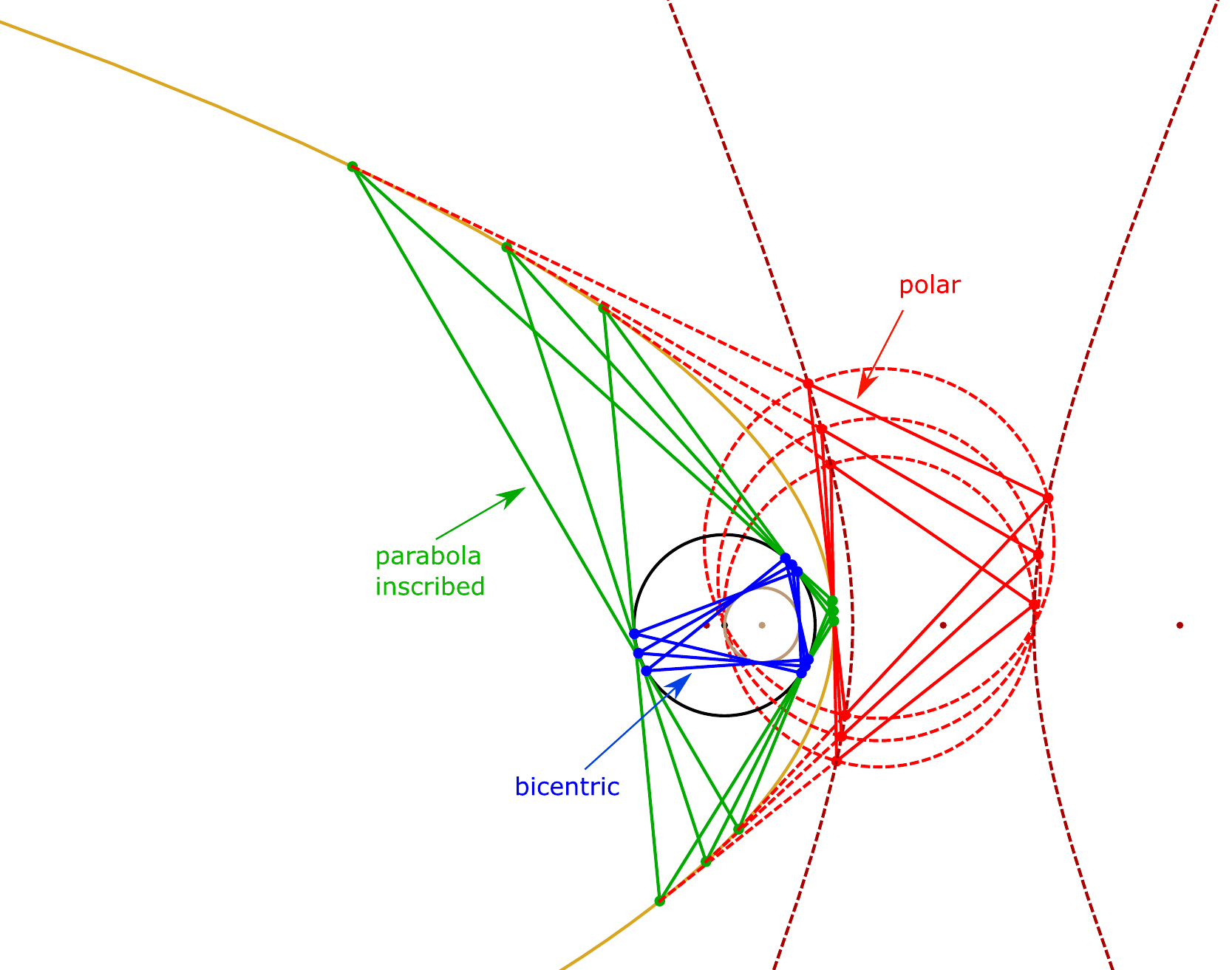}
    \caption{Several configurations of the parabola-inscribed Poncelet family (green), obtainable as the polar image of the bicentric family (blue) with respect to the outer circle (black), provided the bicentric incircle passes through the circumcenter, see \cref{app:bic}. Also shown is the polar family (red) of the parabola-inscribed one with respect to the parabola itself. This family is inscribed in a hyperbola (dark dashed red). So you can think of this trio (blue, green, red) as successive polar images with respect to the outer conic of each preceding family.}
    \label{fig:bic-tang-n3}
\end{figure}

\subsection*{Main results} 

\begin{itemize}
   \item The loci of vertex, perimeter, and area centroids are parabolas. Recall that in general, the locus of the perimeter centroid is not a conic \cite{sergei2016-com}.
   \item The loci of vertex and area centroids of polar polygons are straight lines, whereas that of the perimeter centroid is a non-conic.
   \item In the $N=3$ case, the locus of the orthocenter is a straight line as are those of many triangle centers of the polar family. The Euler line of the polar family always passes through the parabola's focus.
   \item Several centers of the $N=3$ polar family are stationary and/or sweep circles. In the latter case, they all belong to a single parabolic pencil.
   \item We prove that the quantity $\sum{\sin\theta_i/2}$ is conserved, where $\theta_i$ are the interior angles of parabola-inscribed polygons. In fact, this quantity is conserved by any conic-inscribed polar image of the bicentric family.
\end{itemize}

Most of the above properties were first noticed via simulation \cite{mathematica_v10}, and later proved with a computer-algebra system (CAS) \cite{maple2019}, using the explicit parametrizations given in \cref{app:vtx}. For brevity, we omit any CAS-based proofs.

\subsection*{Related work}

We can roughly divide it into three groups: (i) the study of point loci over certain triangle families \cite{odehnal2011-poristic,pamfilos2004,zaslavksy2003}, (ii) proving that loci of certain Poncelet triangle families are of a given curve type \cite{corentin2021-circum,garcia2020-ellipses,olga14,skutin2013-isogonal}, and (iii) proving properties and invariants over $N\geq 3$ Poncelet families \cite{akopyan2020-invariants,bialy2020-invariants,caliz2020-area-product,sergei2016-com}. Also related is the Steiner-Soddy Poncelet family which are the polar image of the so-called Brocard porism with respect to the circumcircle \cite{garcia2022-steiner-soddy}.

\subsection*{Article organization}

In \cref{sec:n3,sec:n3-polar} we examine parabola-inscribed Poncelet triangles (as well as its polar polygons with respect to the parabola), deriving closure conditions and expressions for many of its their triangle center loci. In \cref{sec:n4,sec:n5,sec:n6} we derive geometric closure conditions for  $N=4,5,6$ families, respectively, detecting the abovementioned pattern for the loci of their centroids (as well as in the polar family), followed by conjectured generalizations in \cref{sec:generalize}. In \cref{sec:conserved} we describe a new quantity conserved by the parabola-inscribed family (and variations thereof). 

In \cref{app:vtx} we provide explicit parametrizations for the vertices of both the $N=3$ and $N=4$ families, as well as their respective polar families. In \cref{app:bic} we explore the relation of parabola-inscribed families with the traditional bicentric family.
\section{Loci of parabola-inscribed triangles}
\label{sec:n3}
Referring to \cref{fig:n=3}, consider a Poncelet family $T$ of triangles inscribed in a parabola $\P$, and circumscribed about a focus-centered circle. Let $F=[-f,0]$ and $V=[0,0]$ denote focus and vertex, respectively, where $f$ is the focal distance. Consider a circle $\C$ centered at $F$ with radius $r$. 

\begin{figure}
    \centering
    \includegraphics[width=.8\textwidth,frame]{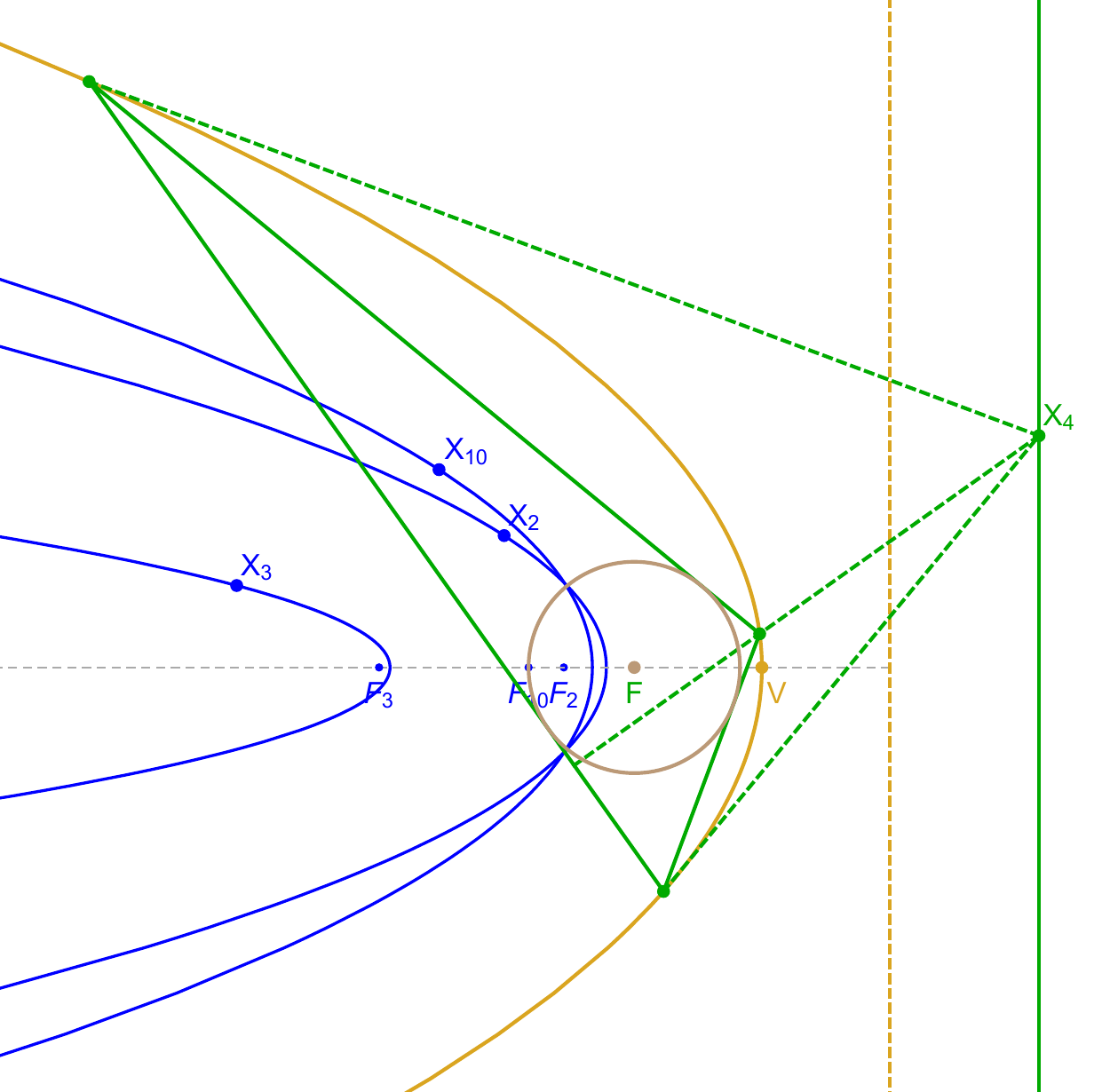}
    \caption{
    A Poncelet triangle (green) is shown inscribed in a parabola $\P$ (gold), circumscribed about a focus-centered circle (brown). Over the family, $X_4$ sweeps a line (solid green) parallel to the directrix (dashed gold). The loci of barycenter $X_2$, circumcenter $X_3$, and Spieker center $X_{10}$ are coaxial parabolas (blue) with foci on $F_2$, $F_3$, and $F_{10}$, respectively. Notice the latter coincides with the left extreme of the circular caustic.}
    \label{fig:n=3}
\end{figure}

\begin{proposition}
$\P$ and $\C$ will admit a Poncelet family of triangles iff $r/f=2(\sqrt{2}-1)$.
\label{prop:n=3}
\end{proposition}

\begin{proof}
Consider the Poncelet triangle with two parallel sides shown in \cref{fig:n=3-proof}, inscribed in the parabola $y=x^2/(4f)$, where $f$ is the focal length. At $x=r$ the parabola must be at $y=f-r$, i.e., $f-r=r^2/(4f)$, and the result follows. 
\end{proof}

\begin{figure}
    \centering
    \includegraphics[width=.5\textwidth,frame]{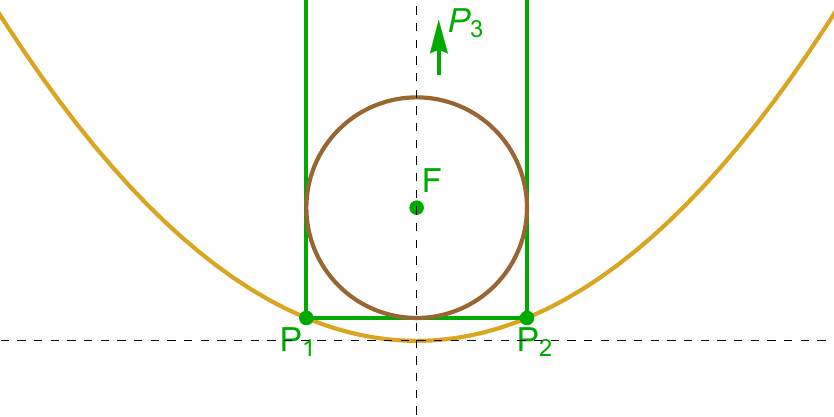}
    \caption{Construction used to derive $r/f$ in \cref{prop:n=3}. Side $P_1 P_2$ of the Poncelet triangle is perpendicular to the axis while the other two sides are parallel to it, i.e., vertex $P_3$ lies on the line at infinity.}
    \label{fig:n=3-proof}
\end{figure}

\subsection{Straight-line orthocenter locus} let $\T$ denote our parabola-inscribed triangle family and $\D$ the directrix of $\P$. Henceforth we shall adopt Kimberling's notation $X_k$ to refer to triangle centers \cite{etc}.

\begin{proposition}
Over $\T$, the locus of the orthocenter $X_4$ is the line parallel to $\D$ given by $x = (5-2\sqrt{2})f$, with $y\in\mathbb{R}\backslash\{\pm 2f/(\sqrt{2}+1)\}$.
\label{prop:x4}
\end{proposition}

The proof below was kindly contributed by Alexey Zaslavsky \cite{zaslasvsky2021-private}:

\begin{proof}
Let $\C$ be the unit circle in the complex plane and $A$, $B$, $C$ the touching points with the sides of the parabola-inscribed triangles. The polar transformation with center $F$ maps the parabola to a circle with center $I$ passing through $F$ and touching $AB$, $BC$, $CA$. Using Euler's formula $|FI|^2=r^2=R(R-2r)$ \cite{mw}, with $R=1$ its radius, and $r=\sqrt{2}-1$. Consider the line $FI$ as the real axis. Since $I$ is self-conjugated with respect to $ABC$, we have $a+b+c=2\sqrt{2}-2+(3-2\sqrt{2})abc$, $ab+bc+ca=3-\sqrt{2}+(2\sqrt{2}-2)abc$. The polar images of the altitudes of the original triangle are the common points of $BC$, $CA$, $AB$ with the lines passing through $F$, and perpendicular to $FA$, $FB$, $FC$ respectively. We have to calculate the common point of the line passing through these three points and the real axis. The coordinate functions of this point are symmetric functions in $a,b,c$, so we can express them as elementary symmetric functions on said variables, and verify that they are constant.  
\end{proof}

In \cref{app:bic} we describe how the parabola-inscribed family is the polar image of the bicentric family with respect to its circumcircle. Referring to \cref{fig:x4-loci}, \cref{prop:x4} is actually a special case of:

\begin{figure}
    \centering
    \includegraphics[width=\textwidth]{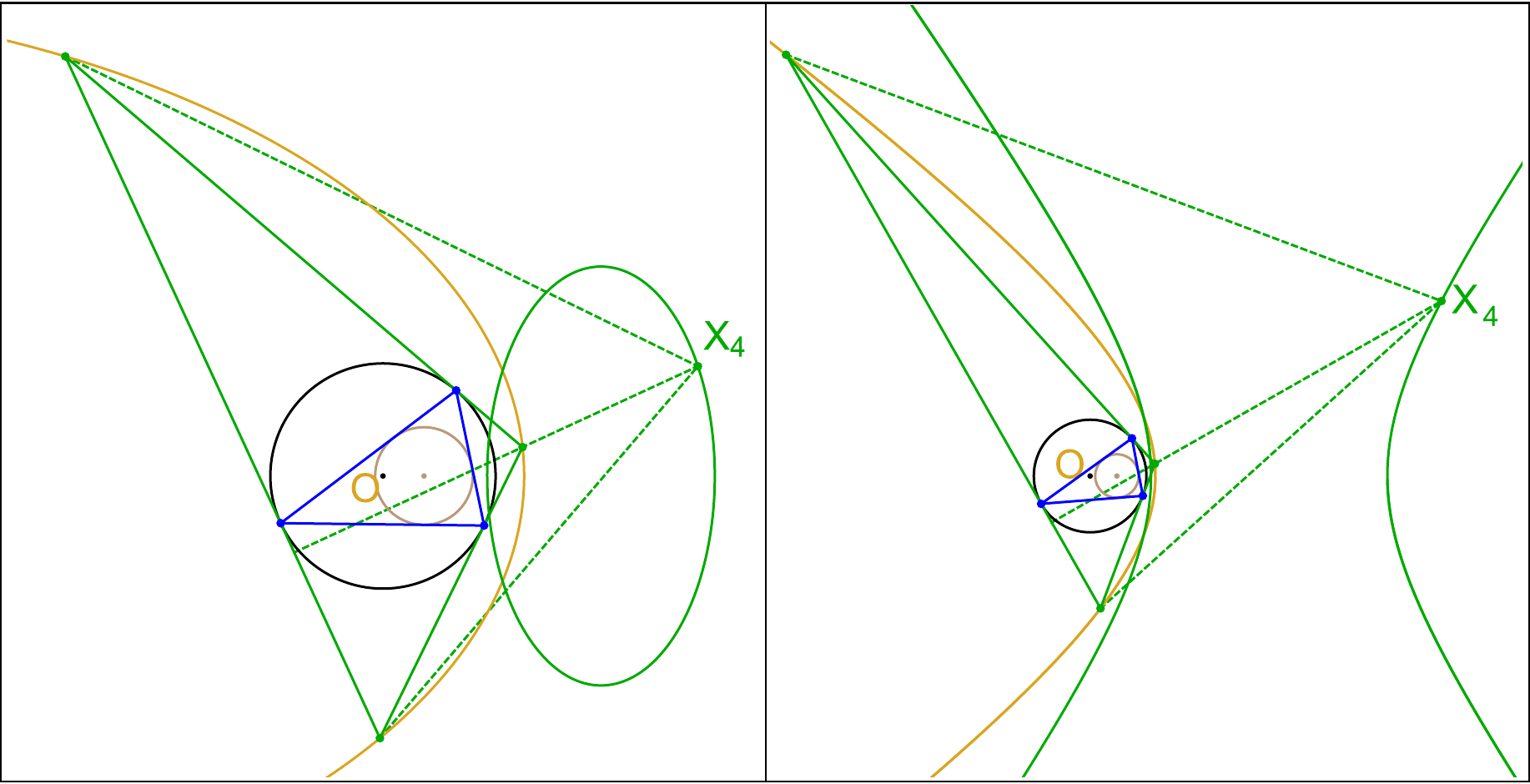}
    \caption{Consider perturbing to the bicentric family (blue) such that the circumcenter $O$ is interior (resp. exterior) to the incircle, as shown on the left (resp. right). The tangential family (green) becomes ellipse- (resp. hyperbola-) inscribed (gold curve). In the former (resp. latter) case, the locus of the orthocenter $X_4$ is an ellipse (resp. hyperbola).}
    \label{fig:x4-loci}
\end{figure}

\begin{proposition}
The locus of $X_4$ of an family which is the polar image of $N=3$ bicentrics with respect to its outer circle is an ellipse, straight line, or hyperbola if the circumcenter of the bicentric triangle lies in the interior, on top, or outside its incircle.
\end{proposition}

\subsection{Three parabolic loci}

Referring to \cref{fig:n=3}, we show below that over $\T$, the loci of the barycenter, circumcenter, and Spieker centers are all parabolas. The first and last correspond to the vertex and perimeter centroids of a triangle. This is curious since, in general, the locus of the perimeter centroid of a Poncelet family is not a conic \cite{schwartz2016-com}.

\begin{proposition}
Over $\T$, the locus of the barycenter $X_2$ is a parabola coaxial with $\P$, with focus $F_2=[-f/3,0]$, and vertex
$V_2=[2f(1-2\sqrt{2})/3,0]$.
\end{proposition}

\begin{proposition}
Over $\T$, the locus of the circumcenter $X_3$ of $T$ is a parabola coaxial with $\P$, with focus $F_3=[-f(2\sqrt{2} - 3)/2,0 ]$, and vertex 
$V_3=[ -f(2\sqrt{2}+3)/2,0]$.
\end{proposition}
 
\begin{proposition}
Over $\T$, the locus of the Spieker center $X_{10}$ is a parabola coaxial with $\P$, with focus $F_{10}=[(1-2\sqrt{2})f,0]$ and vertex  $V_{10}=[f(3/2-2\sqrt{2}),0]$. In particular, $F_{10}=[-f-r,0]$, i.e., it lies on the left extreme of $\C$.
\end{proposition}


\section{The polar \torp{$N=3$}{N=3} family}
\label{sec:n3-polar}
Referring to \cref{fig:polar}, let $T'$ denote the {\em polar} triangle of a triangle $T$ in $\T$, i.e., whose sidelines are the polars of $T$ with respect to $\P$. Since $T$ is inscribed in $\P$ these are simply the tangents.

\begin{figure}
    \centering
    \includegraphics[trim=200 100 50 150,clip,width=\textwidth,frame]{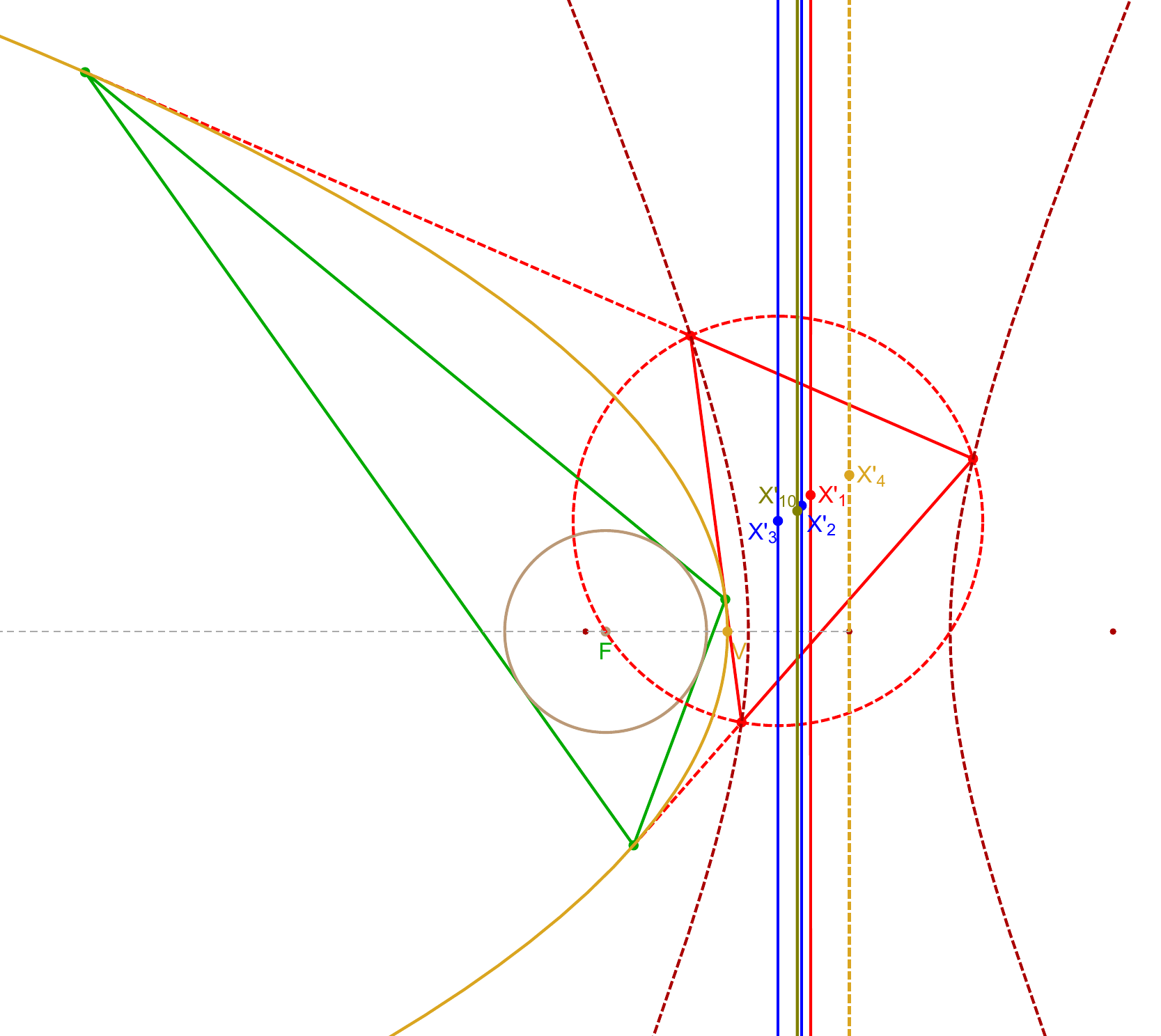}
    \caption{The polar triangle $T'$ (red) with respect to the parabola $\P$ (gold) to which our Poncelet family (green) is inscribed. It is Ponceletian as it is inscribed in a hyperbola (dashed dark red). Well-know properties include (i) the circumcircle (dashed red) passes through the focus $F$, and (ii) the orthocenter $X_4'$ lies (and therefore sweeps) the directrix of $\P$ \cite{akopyan2007-conics}. Also shown are the visually-straight, though quartic loci of the polar incenter $X_1'$ and Spieker center $X_{10}'$ (red and olive, respectively). The loci of $X_k'$, $k=2,3,4$ are straight lines parallel to the directrix (red, red, and dashed gold), the latter the directrix itself.}
    \label{fig:polar}
\end{figure}

Recall some known properties of the polar triangle with respect to any parabola \cite{akopyan2007-conics}: (i) the circumcircle of $T'$ passes through the focus $F$; (ii) the orthocenter of $T'$ is on the directrix; (iii) its area is half that of the reference triangle.

\begin{proposition}
The $T'$ family is Ponceletian. It is circumscribed about $\P$ and is inscribed in a hyperbola $\H$ with center $[f,0]$. Its axes are the axis and directrix of $\P$. Its implicit equation reads

\[ \H:\;\left(\sqrt{2} + \frac{3}{2}\right)\left(x - f\right)^2 - \frac{y^2}{2} - 2f^2=0. \]
\end{proposition}

\subsection{Straight and nearly-straight loci}

An enduring conjecture has been that the locus of the incenter $X_1$ of a Poncelet triangle family can only be a conic if the pair is confocal \cite{helman2021-power-loci}. 

As shown in \cref{fig:polar}, over the polars, the locus of the incenter is, to the naked eye, a straight line. However, upon an algebraic investigation:
 
\begin{proposition}
The locus of the incenter $X_1'$ of $T'$ is one of four branches of the following quartic:

\begin{align*} 
X_1':&(-5  \sqrt{2} - 6) x^2 y^2 + (4  \sqrt{2} + 2) f^2 x^2 + (10 \sqrt{2} + 12) f x y^2\\
+&  (8  \sqrt{2} + 4) f^3 x + (3  \sqrt{2} - 16) f^2 y^2 - 14 f^4=0.
\end{align*}
Specifically, the branch
\[X_1'=\;\left[\frac{ \sqrt{2} y^2 + 2 + 2 y^2 - \sqrt{-4 y^2 + 4 y^4 + 8  \sqrt{2} + 8  \sqrt{2} y^2 - 2  \sqrt{2} y^4}}{ \sqrt{2} y^2 - 2 + 2 y^2},y\right], \]
where $y\neq \pm \sqrt{2-\sqrt{2}}$.
\end{proposition}

The locus of $X_1'$ is bounded by two lines parallel to the directrix and approximately $f/850$ apart, see \cref{fig:polar-x1-compressed}. 


\begin{figure}
    \centering
    \includegraphics[width=.7\textwidth]{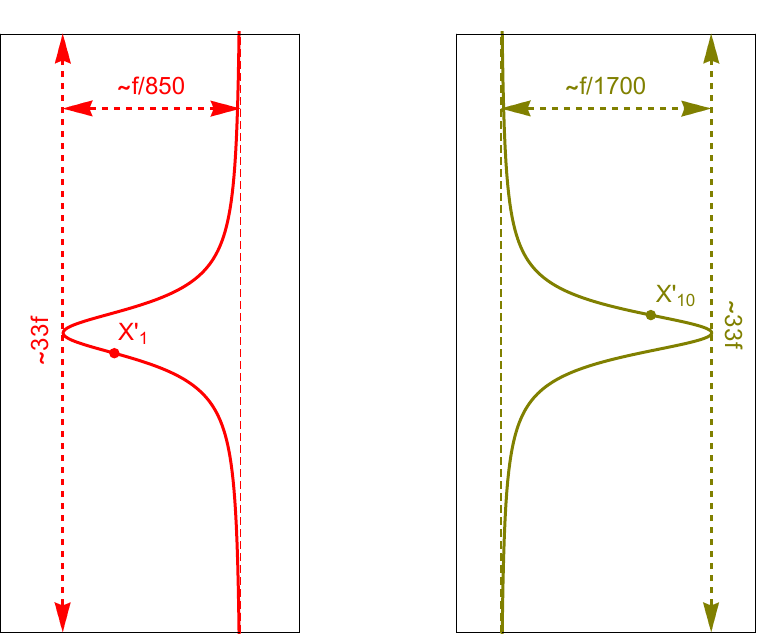}
    \caption{\textbf{Left:} The locus of the polar incenter $X_1'$ is the branch of a quartic which visually is a straight line. It fits within to lines parallel to the directrix and at a distance of $f/850$. In the figure the curve is shown at aspect ratio of 28,000. \textbf{Right:} The locus of the polar Spieker center $X_{10}'$ (perimeter centroid) is an algebraic curve of degree at least four, bounded by two vertical lines separated by $f/1700$. The aspect ratio of the figure is 56,000.}
    \label{fig:polar-x1-compressed}
\end{figure}

Still referring to \cref{fig:polar}:

\begin{proposition}
The locus of the barycenter $X_2'$ of $T'$ is a line parallel to $\D$ and parametrized by
\[
X_2'=\;\frac{1}{3}\left[\left(2 \sqrt{2}-1\right)f,\frac{( 4-8 \sqrt{2} ) f^2 y + y^3}{(8 \sqrt{2} - 12)f^2 + y^2}\right].
\]
\end{proposition}

\begin{proposition}
The locus of the circumcenter $X_3'$ of $T'$ is a line parallel to $\D$ and parametrized by
\[ X_3'=\;\left[(\sqrt{2} - 1) f,\frac{ (3 \sqrt{2} + 2) (2 \sqrt{2} y^2 - 28 f^2 + y^2) y}{14 (y \sqrt{2} - 2 f + y) (y \sqrt{2} + 2 f + y)}\right].
\]
\end{proposition}

\noindent Referring to \cref{fig:polar-stationary}, the above expressions for $X_2'$ and $X_3'$ yield:

\begin{corollary}
The (varying) Euler line $X_2' X_3'$ of the polar family passes through the focus $F=[-f,0]$ of $\P$. \end{corollary}

Still referring to \cref{fig:polar-stationary}, the next 4 propositions were obtains from experimental evidence and verification by CAS:

\begin{proposition}
The locus of the symmedian point $X_6'$ of $T'$ is a line parallel to $\D$ and parametrized by
\[X_6'=\;
\left[(5 - 3 \sqrt{2}) f, \frac{(3 \sqrt{2} + 4) (2 \sqrt{2} y^2 - 28 f^2 + y^2) y}{14 (y \sqrt{2} - 2 f + y) (y \sqrt{2} + 2 f + y)}\right]
.\]
\end{proposition}

\begin{proposition}
The locus of $X_{10}'$ of the polar family is an algebraic curve of degree  four given by
{\small
\begin{align*}
  X_{10}':\; &4\left( 11\,\sqrt {2}+16 \right) {x}^{4}- 4\left(  3\,\sqrt {2}+5 \right) {x}^{2}{y}^{2}-4\, \left( 37\,\sqrt {2}+50
 \right) f{x}^{3}+8\, \left( 2\,\sqrt {2
}+1 \right) fx{y}^{2}\\
&+21
\, \left( 5\,\sqrt {2}+8 \right) {f}^{2}{x}^{2} -4\, \left( 9\,\sqrt {2}+8 \right) {f}^{3}x-
 \left( \sqrt {2}+4 \right) {f}^{2}{y}^{2}+7\,{f}^{4}
 =0
 \end{align*}
 }

This locus is tightly bound by the following two lines parallel to the directrix:

\[x=\left(\sqrt {2} -1 + \frac{\sqrt{10 - 7\sqrt{2}}}{2}  \right) f  
\;\;\textrm{and}\;\;
x=\left(\sqrt{2} - 2^{-1/4}
 \right) f.
\]
The distance between these lines is   approx. $f/1700$.
\end{proposition}

\subsection{Stationary points}

\begin{figure}
    \centering
    \includegraphics[trim=100 0 75 25,clip,width=\textwidth,frame]{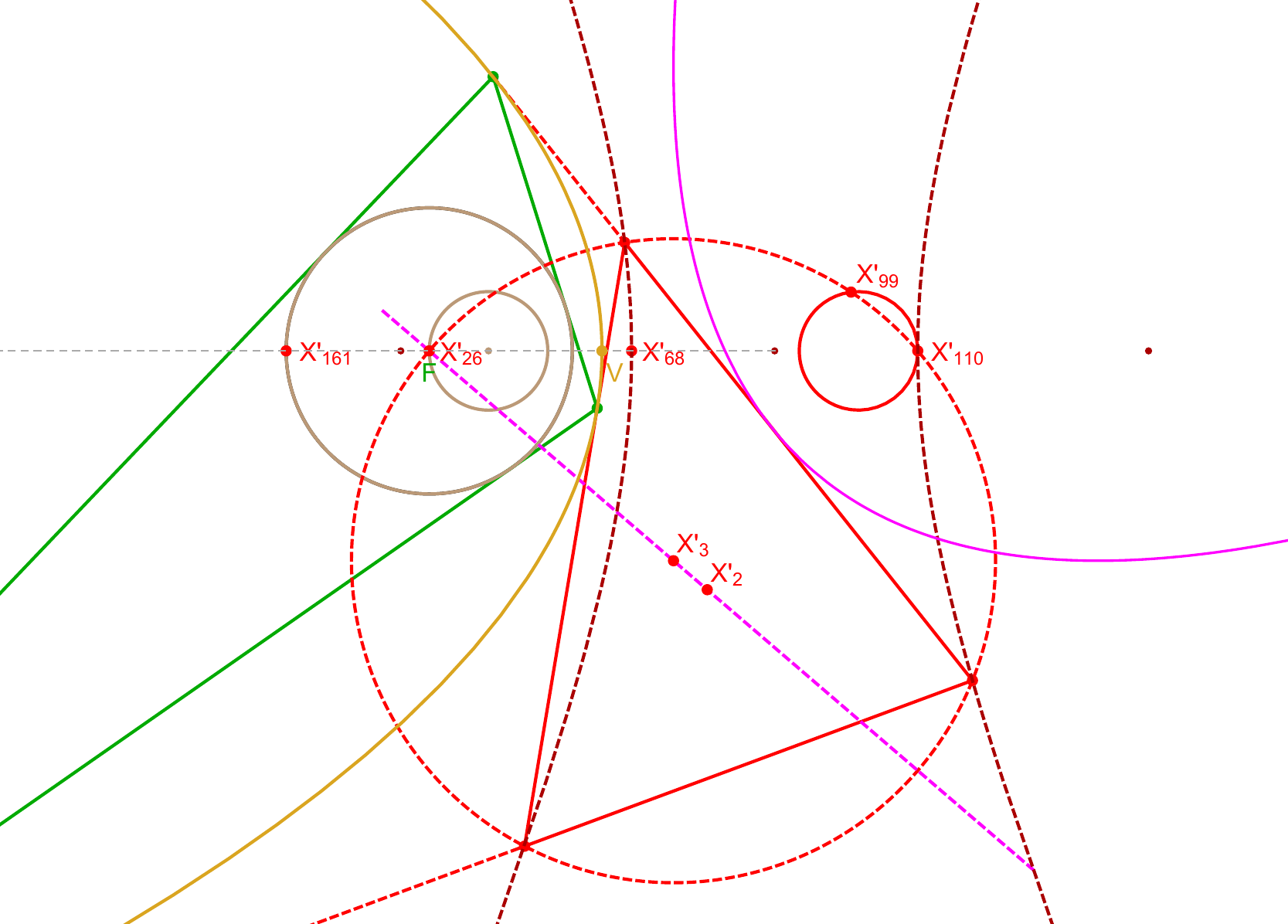}
    \caption{Over the polar family (red), the Euler line (dashed magenta) will always pass through the focus $F$ of the parabola-inscribed family (green). $X_{26}'$ (resp. $X_{68}'$ and $X_{110}'$) remain stationary at the focus $F$ (resp. the two vertices of the hyperbola to which the polar family is inscribed). Experimentally, $X_{161}$ is stationary at the intersection of the caustic with the parabola axis farthest from the latter's vertex. Also shown is the Kiepert inparabola (magenta), whose focus is $X_{110}$ and directrix is the Euler line. Thus the polar family simultaneously inscribes the original parabola (gold) and the Kiepert (magenta). Finally, the figure depicts the circular locus of Steiner point $X_{99}'$ of the polar family.}
    \label{fig:polar-stationary}
\end{figure}

The circumcenter of the tangential triangle appears as $X_{26}$ on \cite{etc}. 

\begin{proposition}
Point $X_{26}'$ of $T'$ is stationary at the focus $F$ of $\P$.
\end{proposition}

Note $X_{26}$ does not lie in general on the circumcircle of a reference triangle. In our case it does since, as mentioned above, the circumcircle of the polar contains the focus.

The Kiepert parabola of a triangle is an inscribed parabola whose focus is labeled $X_{110}$ on \cite{etc}. Its directrix is the Euler line \cite{mw}. Referring to \cref{fig:polar-stationary}:

\begin{proposition}
The focus $X_{110}'$ of the Kiepert parabola (resp. the Prasolov point $X_{68}'$) of the polar family is stationary at the vertex of $\H$ farthest (resp. closest) to the focus of $\P$. Furthermore, $X_{161}'$ is stationary at the intersection of the incircle with the parabola axis farthest from the parabola vertex, i.e., at  $[(1-2\sqrt{2})f,0]$. 
\end{proposition}

\begin{observation}
Over the polar family, the vertex of its Kiepert parabola sweeps a circle.
\end{observation}

\subsection{Linear loci galore} Referring to \cref{fig:polar-barcode}:

\begin{figure}
    \centering
    \includegraphics[trim=70 0 20 0,clip,width=\textwidth,frame]{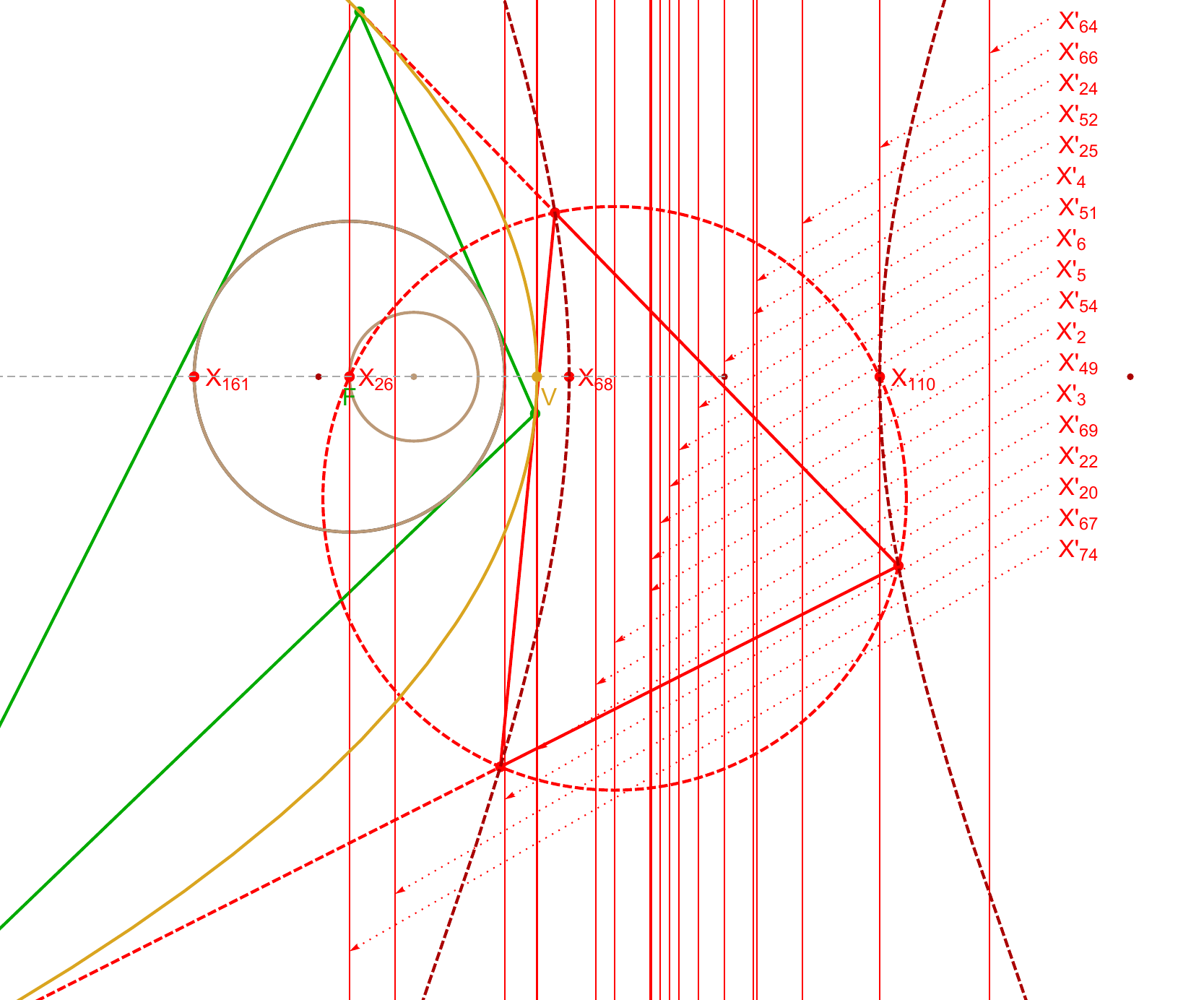}
    \caption{Many triangle centers of the polar family sweep lines parallel to the directrix. The following are shown: $X_k$, $k=$2, 3, 4, 5, 6, 20, 22, 24, 25, 49, 51, 52, 54, 64, 66, 67, 69, 74.}
    \label{fig:polar-barcode}
\end{figure}

\begin{observation}
Over the first 1000 triangle centers in \cite{etc}, the following triangle centers of $T'$ sweep linear loci parallel to $\D$: $X_k', k=$2, 3, 4, 5, 6, 20, 22, 23, 24, 25, 49, 51, 52, 54, 64, 66, 67, 69, 74, 113, 125, 140, 141, 143, 146, 154, 155, 156, 159, 182, 184, 185, 186, 193, 195, 206, 235, 265, 323, 343, 368, 370, 373, 376, 378, 381, 382, 389, 394, 399, 403, 427, 428, 468, 546, 547, 548, 549, 550, 567, 568, 569, 575, 576, 578, 597, 599, 631, 632, 858, 895, 973, 974. 
\end{observation}

\subsection{A pencil of circular loci} Referring to \cref{fig:polar-stationary}:

\begin{proposition}
The locus of the Steiner point $X_{99}'$ is a circle whose center $O_{99}'$ lies on the axis of $\P$ of radius $R_{99}'$ such that at its right endpoint it touches $X_{110}'$. Explicitly,
\[ 
O_{99}'=\left[ (6\sqrt{2} - 7) f,0\right],\;\;\;R_{99}'= 2f \sqrt{17-12\sqrt{2}} \; \cdot  \]
\end{proposition}

\noindent Referring to \cref{fig:polar-locus-pencil}:

\begin{observation}
Over the first 1000 triangle centers in \cite{etc}, the following triangle centers of $T'$ sweep circular loci with centers on the axis of $\P$ and passing through $X_{110}'$: $X_k', k=$99, 107, 112, 249, 476, 691, 827, 907, 925, 930, 933, 935.
\end{observation}

This gives credence to:

\begin{conjecture}
If the locus of $X_k'$ is a circle with nonzero radius, it is in the parabolic pencil with $X_{110}$ as a common point. 
\end{conjecture}

\begin{figure}
    \centering
    \includegraphics[trim=50 20 100 100,clip,width=\textwidth,frame]{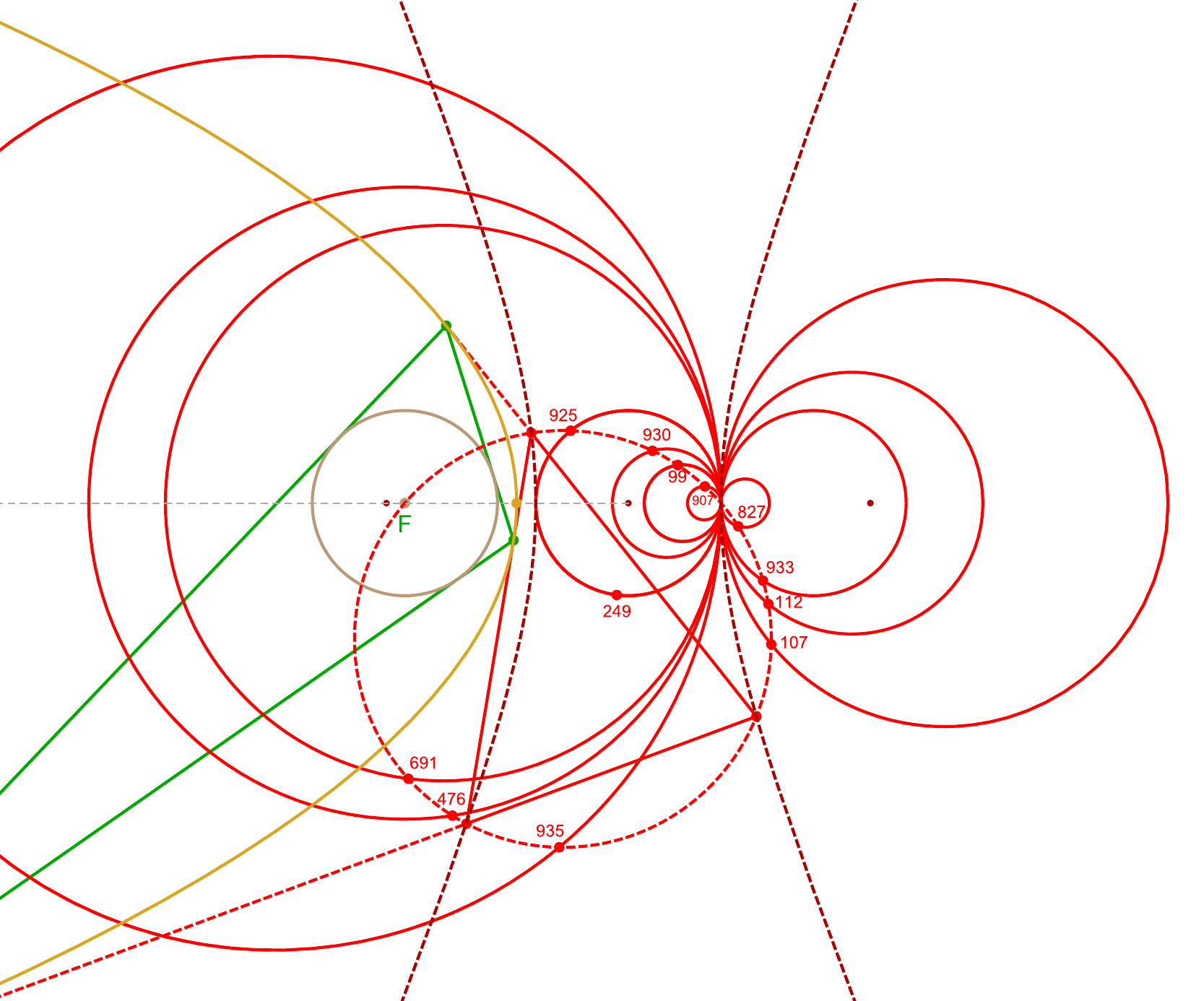}
    \caption{Over the polar family we find that if a certain triangle center sweeps a circular locus, said locus will be an element of a parabolic pencil with $X_{110}$ as their common point (not labeled). In the figure the circular loci of $X_k$, $k=$99, 107, 112, 249, 476, 691, 827, 907, 925, 930, 933, 935 are shown. Notice all lie on the dynamically-moving circumcircle (dashed red) except for $X_{249}$.}
    \label{fig:polar-locus-pencil}
\end{figure}

\section{Parabola-inscribed quadrilaterals}
\label{sec:n4}
Referring to \cref{fig:n=4}, consider a Poncelet family $\Q$ of quadrilaterals inscribed in a parabola $\P$, and circumscribed about a focus-centered circle $\C$ of radius $r$. As before, let $f$ denote the parabola's focal distance, and $V=[0,0]$, $F=[-f,0]$, its vertex and focus, respectively.

\begin{figure}
    \centering
    \includegraphics[trim=50 50 0 150,clip,width=\textwidth,frame]{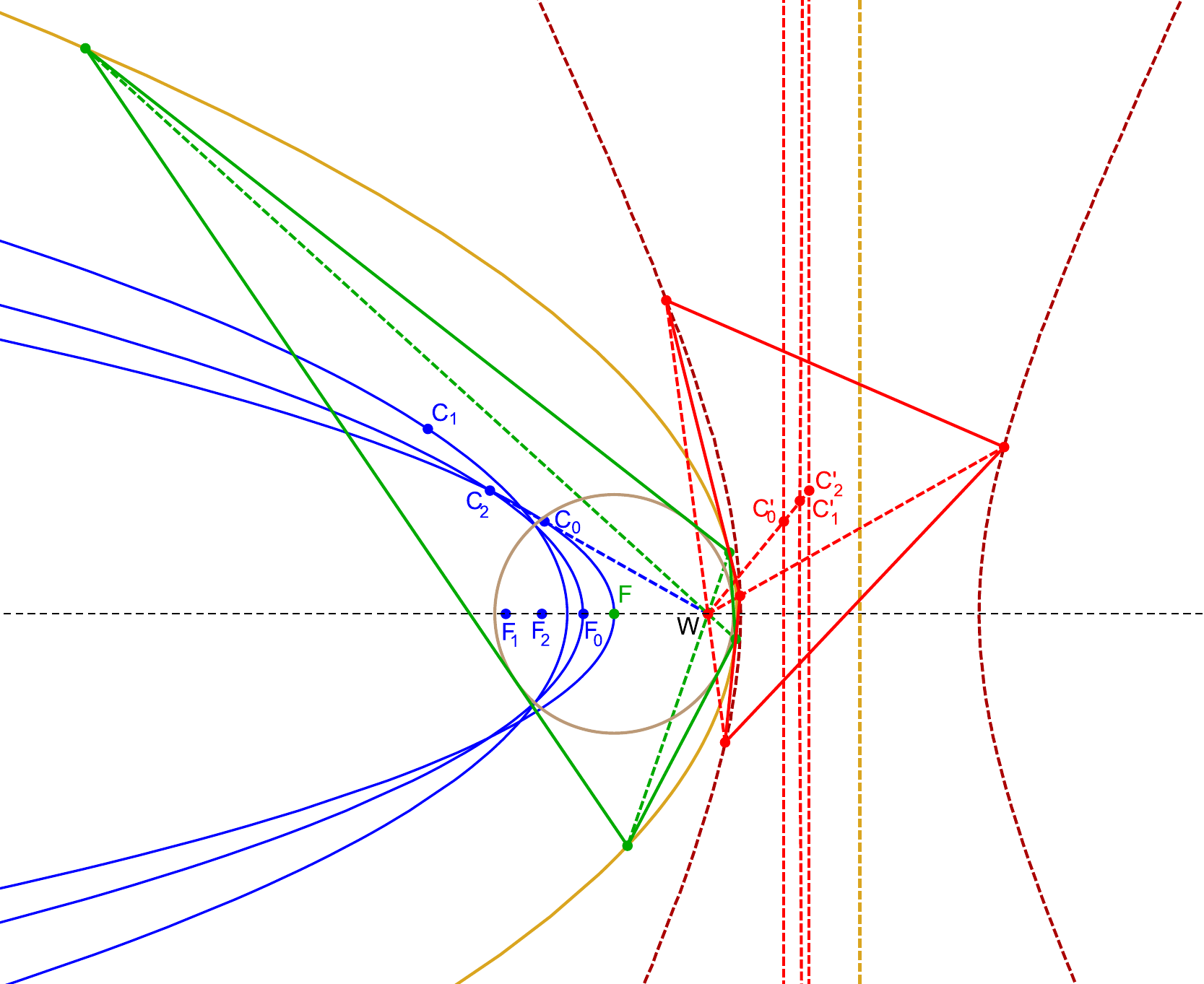}
    \caption{A Poncelet quadrilateral (green) is shown inscribed in a parabola $\P$ (gold) and circumscribed about a focus-centered circle (brown). Over the family, (i) the intersection $W$ of its diagonals (dashed green) is stationary; (ii) the loci of vertex $C_0$, perimeter $C_1$, and area $C_2$, centroids sweep 3 distinct parabolas (blue) coaxial with $\P$ with foci on $F_0$, $F_1$ and $F_2$. Notice the vertex of $C_0$ is $F$ and that of $C_1$ is $F_0$. (iii) $C_0,C_2,W$ are collinear (dashed blue). Also shown is the polar quadrilateral $Q'$ (red) with respect to $\P$, inscribed in a hyperbola (dashed, red) centered at $[f,0]$. One observes that: (i) its diagonals (dashed red) also intersect at $W$; (ii) the loci of its vertex $C_0'$ and area $C_2'$ centroids are lines (dashed orange) perpendicular to the axis of $\P$, (iii) $C_0',C_2',W$ are collinear (dashed red); (iv) the locus of the polar perimeter centroid $\C_1'$ is algebraic and of degree 10.}
    \label{fig:n=4}
\end{figure}

\begin{proposition}
$\P$ and $\C$ will admit a Poncelet family of convex quadrilaterals iff  $r/f=2\sqrt{\sqrt{5}-2}$.
\label{prop:n=4}
\end{proposition}

\begin{proof}
Referring to \cref{fig:n=4-proof}, consider the symmetric Poncelet quadrilateral $P_i=[x_i,y_i]$, $i=1,\ldots,4$, inscribed in the parabola $y=x^2/(4f)$, i.e., $x=2\sqrt{f y}$. Clearly, $y_1=f-r$, and $y_2=f+r$. Requiring that $P_1 P_2$ be tangent to $\C$ yields the quartic $r^2 + 4 f\sqrt{f^2 - r^2}=0$. The claim is the one positive root of this quartic.
\end{proof}

Note: more generally, Cayley's conditions may be used to include the non-convex case, see \cite{dragovic11}.

\begin{figure}
    \centering
    \includegraphics[width=.5\textwidth,frame]{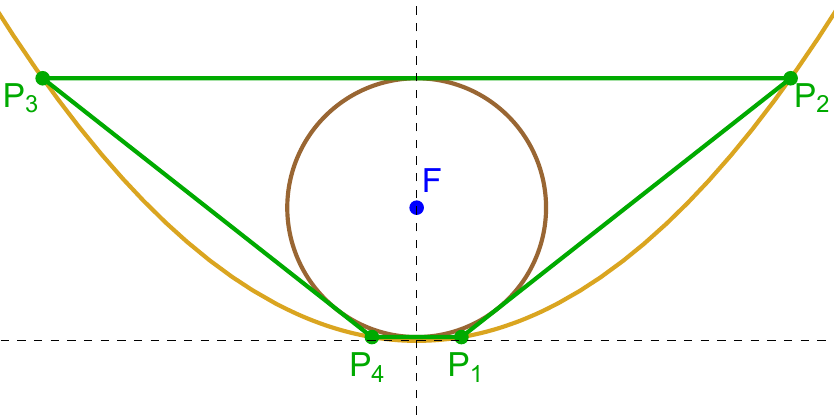}
    \caption{Construction used to derive $r/f$ in for parabola-inscribed convex quadrilaterals in \cref{prop:n=4}.}
    \label{fig:n=4-proof}
\end{figure}

The next 3 propositions, first identified experimentally, were then confirmed via CAS.

\begin{proposition}
Over $\Q$, the two major diagonals $P_1 P_3$ and $P_2 P_4$ intersect at a stationary point $W=[(2 - \sqrt{5})f, 0]$.
\label{prop:n4w}
\end{proposition}

\subsection{The three centroids}

Referring to \cref{fig:n=4}, let $C_0$, $C_1$, and $C_2$ denote the vertex, perimeter, and area centroids of the quadrilaterals in $\Q$, respectively.

\begin{proposition}
Over the family, $C_0$, $C_2$, and $W$ are collinear.
\end{proposition}

\begin{proposition}
Over the Poncelet family, the loci of $C_0,C_1,C_2$ are parabolas coaxial with $\P$, whose foci and vertices locations are listed in \cref{tab:n4}.
\end{proposition}

\begin{table}
\centering
\begin{tabular}{|c|c|l|l|}
\hline
centroid (N=4) & focal dist. & vertex $x/f$ & vtx. $x/f$ (num). \\
\hline
$C_0$ & $f/4$ & $-1$ & $-1$ \\
$C_1$ & $f/2$ &  
$( \sqrt{5}-5)/2$ & 
{$-1.381966$} \\
$C_2$ & $f/3$ &  $\sqrt{5}/3-2$ & $-1.25464$ \\
\hline
\end{tabular}
\caption{Location of centroids $C_0,C_1,C_2$ in the convex $N=4$ family}.
\label{tab:n4}
\end{table}
 
\subsection{The polar quadrilateral}

Referring to \cref{fig:n=4}, consider the polar quadrilateral whose sides are the tangents to $\P$ at the vertices of the original family. Let $P_i'$, $i=1,\ldots,4$ denote its vertices and $C_0'$, $C_1'$, and $C_2'$ denote its vertex, perimeter, and area centroids.

\begin{proposition}
The locus of the polar quadrilateral's vertices is the hyperbola $\H$ given by
\[ \H:\;  \frac{(x-f)^2}{4(    \sqrt{5}-2)f^2}-\frac{y^2}{4f^2}-1=0.\] 
with center at $[f,0]$ and foci $[f(1\pm 2\sqrt{\sqrt{5}  - 1}) ,0]$.
\end{proposition}

Let $W$ be defined as in \cref{prop:n4w}. The next two propositions result from visual (and numerical) detection, followed by verification by CAS.

\begin{proposition}
The two diagonals of the polar quadrilateral intersect at $W$.
\end{proposition}

\begin{proposition}
Over the polar quadrilateral family, $C_0'$, $C_2'$, and $W$ are collinear.
\end{proposition}

\begin{proposition}
Over $\Q$, the loci of $C_0'$ and $C_2'$ are lines parallel to the parabola's directrix and given by $C_0':\; x= (3-\sqrt{5} )f/2$, and $C_2':\; x=  (4-\sqrt{5} )f/3$.

\end{proposition}

Rather laborious CAS manipulation yields:

\begin{proposition}
Over $\Q$, the locus of $C_1'$ is one connected component of an algebraic curve of degree ten, given by the following equation:
{\scriptsize  
\begin{align*}
 C_1':&\;\;  - \left( 1457008\,\sqrt {5}+3257968 \right) {f}{x}^{7}{y}^{2}+ \left( 
122156\,\sqrt {5}+273148 \right) {f}^{2}{x}^{4}{y}^{4}+ \left( 465164
\,\sqrt {5}+1040132 \right) {f}^{2}{x}^{6}{y}^{2}\\
&- \left( 96506\,
\sqrt {5}+215698 \right) {f}^{6}{x}^{2}{y}^{2}
 - \left( 119256\,\sqrt 
{5}+266664 \right) {f}^{3}{x}^{3}{y}^{4}+ \left( 505052\,\sqrt {5}+
1129268 \right) {f}^{5}{x}^{3}{y}^{2}\\
&+ \left( 8564\,\sqrt {5}+19204
 \right) {f}^{7}x{y}^{2}- \left(  881712\,\sqrt {5}+1971568 \right) {x
}^{10} 
 + \left( 43955\,\sqrt {5}+98289 \right) {f}^{4}{x}^{2}{y}^{4}\\
 &+ \left( 24568\,\sqrt {5}+54936 \right) {f}{x}^{5}{y}^{4}-
 \left( 7250\,\sqrt {5}+16210 \right) {f}^{5}x{y}^{4}- \left( 
1274930\,\sqrt {5}+2850838 \right) {f}^{4}{x}^{4}{y}^{2}\\
&+ \left( 
1235568\,\sqrt {5}+2762832 \right) {f}^{3}{x}^{5}{y}^{2}+ \left( 
4457696\,\sqrt {5}+9967712 \right) {f}{x}^{9} 
 - \left( 7787152\,\sqrt {5
}+17412608 \right) {f}^{2}{x}^{8}\\
&+ \left( 5470456\,\sqrt {5}+12232344
 \right) {f}^{3}{x}^{7}- \left(  1690535+755997\,\sqrt {5} \right) {f}
^{4}{x}^{6}- \left( 812098\,\sqrt {5}+1815898 \right) {f}^{5}{x}^{5}\\
&+
 \left( 330322\,\sqrt {5}+738968 \right) {f}^{6}{x}^{4} 
  + \left( 1002+
448\,\sqrt {5} \right) {f}^{6}{y}^{4} -\left( 228\,\sqrt {5}+672
 \right) {f}^{8}{y}^{2}\\
 &- \left(  7300\,\sqrt {5}+16956 \right) {f}^{7}
{x}^{3}+ \left( 2750\,\sqrt {5}+7150 \right) {f}^{9}x- \left(  16145\,
\sqrt {5}+36103 \right) {f}^{8}{x}^{2}\\
&- \left( 84196\,\sqrt {5}+
188268 \right) {x}^{6}{y}^{4}+ \left( 544928\,\sqrt {5}+1218496
 \right) {x}^{8}{y}^{2}-
726\,{f}^{10}=0.
\end{align*}
}
Furthermore, $C_1'$ is bound by the following two lines parallel to the directrix and approximately $f/25$ apart:
$x= \left( 5+\sqrt {2}  - \sqrt {5}\sqrt {2}-\sqrt {5}\right)  f/2
$, and $x=\left( \sqrt {5}\sqrt {2}-\sqrt {5}-2\,\sqrt {2}+3 \right) f/2$.
\end{proposition}

\section{Parabola-inscribed pentagons}
\label{sec:n5}
Referring to \cref{fig:n=5}, consider a family of pentagons inscribed in a parabola $\P$ of focal distance $f$, and circumscribed about a focus-centered circle $\C$ of radius $r$.

\begin{figure}
    \centering
    \includegraphics[trim=100 50 25 200,clip,width=.9\textwidth,frame]{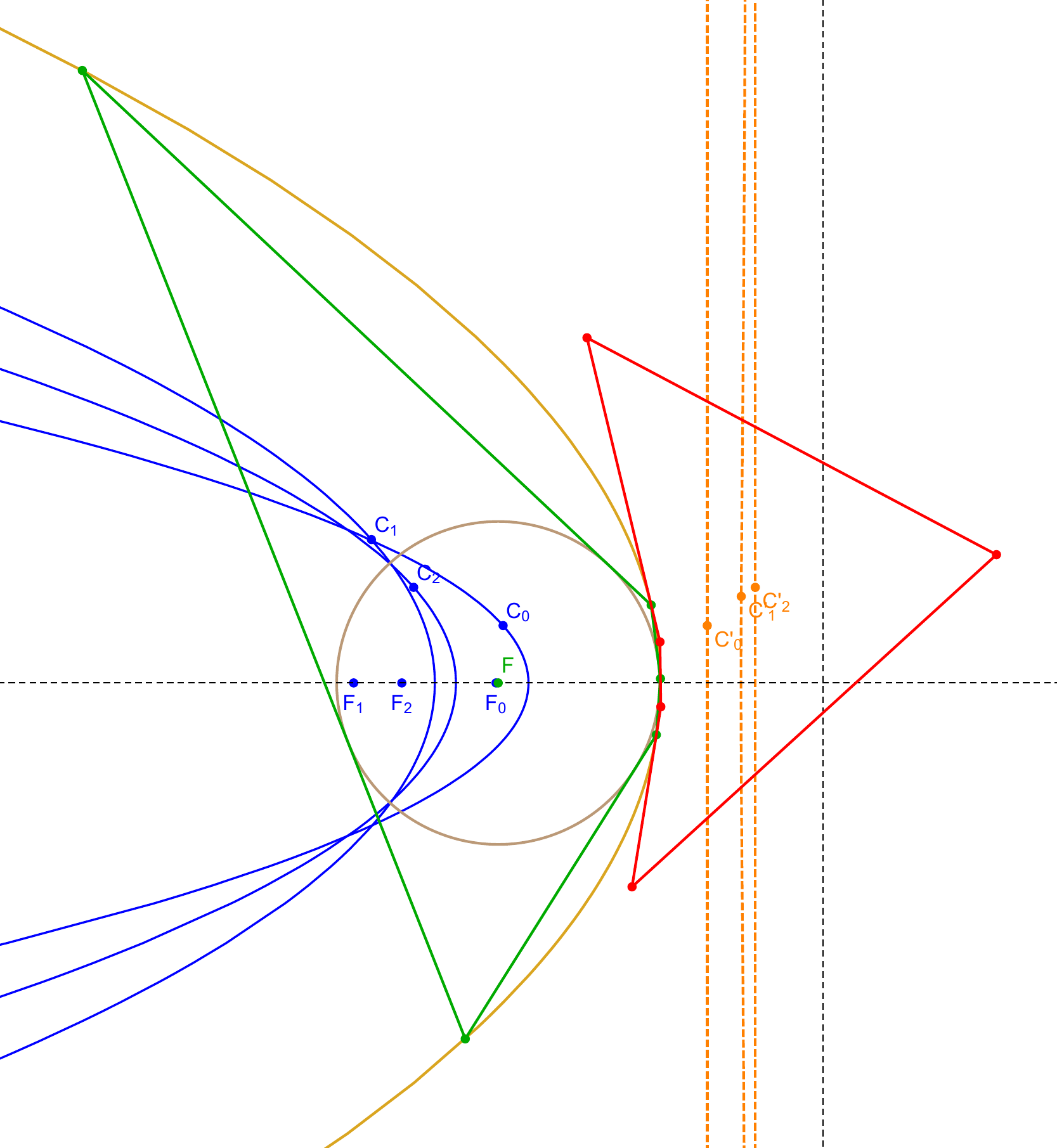}
    \caption{Parabola-inscribed pentagons (green), and their polar polygon (red). The loci of vertex $C_0$, perimeter $C_1$, and area centroids $C_1$ are parabolas (blue) coaxial with $\P$ (gold). Over the polar family, $C_0'$ and $C_2'$ are straight lines (dashed orange) perpendicular to the directrix (dashed black). Though the locus of the perimeter centroid $C_1'$ is indistinguishable from a straight line, it is an algebraic curve of degree likely much higher than 10 (since that is the degree for $C_1'$ on $N=4$).}
    \label{fig:n=5}
\end{figure}

\begin{proposition}
The pair $\P,\C$ will admit a Poncelet family of pentagons iff $r/f$ is the only positive root of the following sextic polynomial ($r/f\approx 0.995219$):
\[ x^6+ 12 x^5- 28 x^4 + 32 x^3+ 112 x^2 - 64 x -64 = 0. \]
\label{prop:n=5}
\end{proposition}
\begin{proof}
Referring to \cref{fig:n=5-proof}, without loss of generality, let $\P$ be the unit parabola $y=x^2$ with focus $F=[0,1/4]$ and let $\C$ be a circle of radius $r$ centered at $F$. Consider the Poncelet pentagon $P_i$, $i=1,\ldots,5$ with $P_4$ at infinity, and $P_1 P_2$ horizontal and tangent to $\C$ at $[0,1/4-r]$. Compute the next Poncelet vertex $P_3=[x_3,y_3]$ as the intersection of a tangent to $\C$ from $P_2$ with $\P$. By requiring that $x_3=r$, we obtain the sextic in the claim.
\end{proof}

\begin{figure}
    \centering
    \includegraphics[width=.7\textwidth]{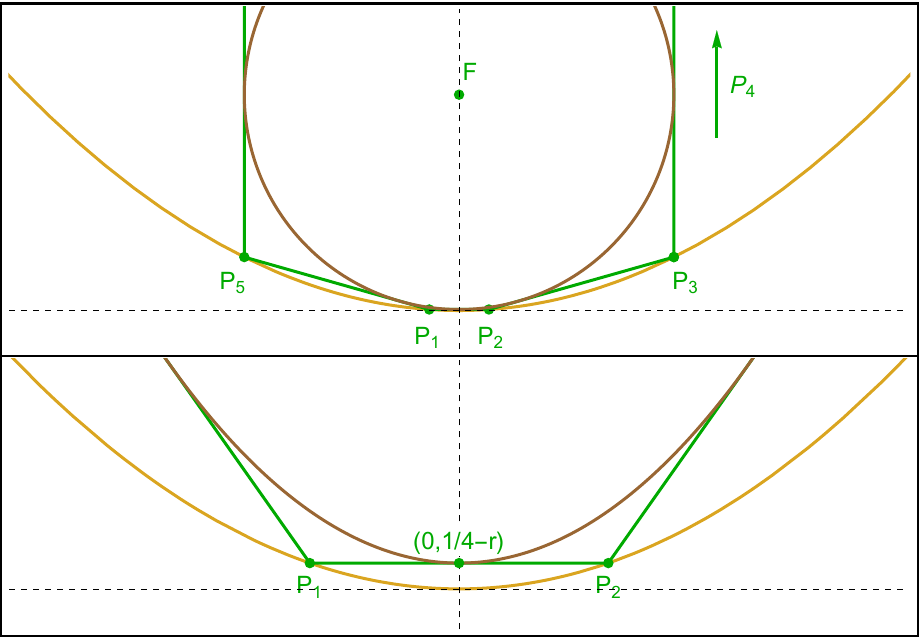}
    \caption{Construction used to derive $r/f$ in \cref{prop:n=5}. Top (resp. bottom) shows the complete picture (resp. a detailed view near the vertex)}
    \label{fig:n=5-proof}
\end{figure}

\noindent Referring to \cref{fig:n=5}:

\begin{conjecture}
Over the parabola-inscribed pentagon family, the loci of vertex, perimeter, and area centroids are parabolas coaxial with $\P$.
\end{conjecture}

\begin{conjecture}
Over the family of polar polygons to parabola-inscribed pentagons, the locus of vertex and area vertices are lines perpendicular to the directrix while that of the perimeter centroid is an algebraic curve of degree at least four.
\end{conjecture}

\section{Parabola-inscribed hexagons and summary}
\label{sec:n6}
\subsection{Hexagons and summary}

Referring to \cref{fig:n=6}, we can also consider a family of parabola-inscribed hexagons.

\begin{figure}
    \centering
    \includegraphics[trim=100 50 20 200,clip,width=.9\textwidth,frame]{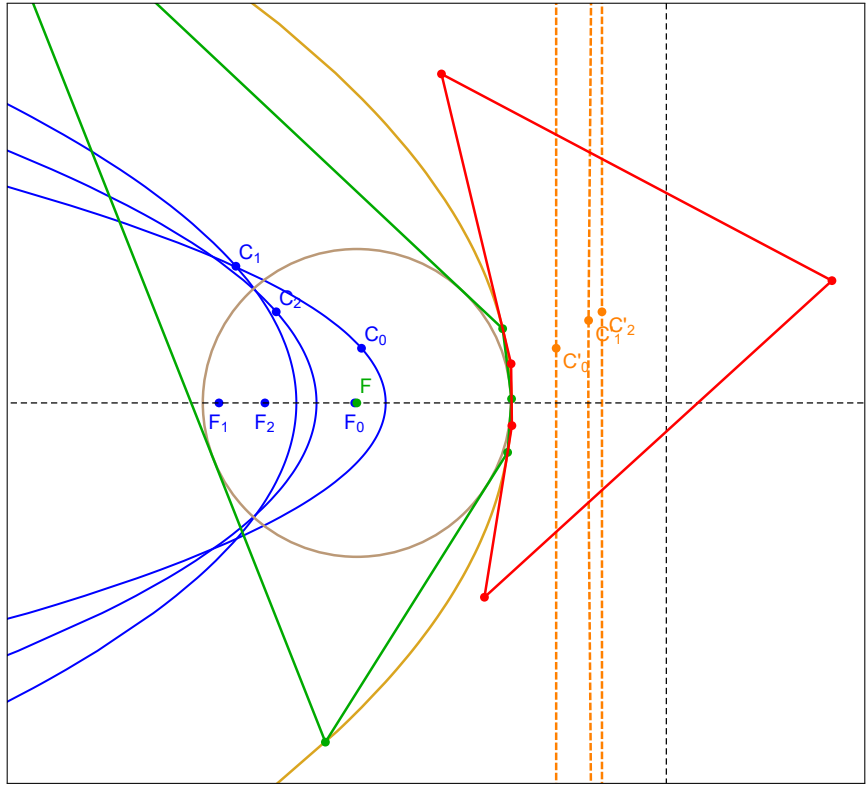}
    \caption{Hexagons (green) inscribed in a parabola $\P$. As before, the loci of $C_0$, $C_1$, and $C_2$ are parabolas (blue) coaxial $\P$. Over the polar family (red), the loci of $C_0',C_2'$ are lines perpendicular to the axis while that of $C_1'$ is algebraic, and though visually a straight line, its degree is likely much higher than 10 (since that is the degree for $C_1'$ on $N=4$).}
    \label{fig:n=6}
\end{figure}

An analogous construction (based on symmetric configurations) was used to obtain $r/f$ required for convex $N=6$. A summary of all $r/f$ thus obtained appears in \cref{tab:rOvF}.

\begin{table}[H]
\begin{tabular}{|c|c|c|c|}
\hline
$N$ & $r/f$ & $r/f$ (num.) & Cayley \\
\hline
3 & $2(\sqrt{2}-1)$ & $0.828427$ & 4 \\
4 & $2\sqrt{\sqrt{5}-2}$ & $0.971737$  & 4 \\
5 & n/a  & $0.995219$ & 8 \\
6 & n/a & $0.999183$ & 8 \\
\hline
\end{tabular}
\caption{Table of $r/f$ required for closure of convex $N$-gons inscribed in a parabola, and circumscribed about a focus-centered circle. Algebraic expressions (2nd column) are only possible for $N=3,4$. The last column shows the number of possible solutions for $r/f$ if one were to include cases where circle and parabola intersect (the Poncelet polygon may be self-intersecting and/or non-convex). For Cayley's conditions in the general case, see \cite{dragovic11}.}
\label{tab:rOvF}
\end{table}

\section{Generalizing centroidal loci}
\label{sec:generalize}
Let $\R$ be a Poncelet family of $N$-gons inscribed to a parabola $\P$, and circumscribed about a focus-centered circle $\C$. Experimental the evidence for the $N=3,4,5,6$ cases, we propose the following generalizations (reader contributions are encouraged):

\begin{conjecture}
Over $\R$, for any $N{\geq}3$, the loci of vertex, perimeter, and area centroids are parabolas coaxial with $\P$.
\label{conj:centroid}
\end{conjecture}

\begin{conjecture}
Over $\R$ and for any $N{\geq}3$, the loci of vertex and area centroids of the polar polygons with respect to $\P$ are straight lines parallel to the directrix of $\P$.
\label{conj:centroid-polar}
\end{conjecture}

Let $\B'$ be the conic-inscribed polar image of a generic bicentric family of $N$-gons with respect to the bicentric circumcircle (see 
 \cref{app:bic}).
 
Recall that the locus of vertex and area centroids $C_0,C_2$ are conics over any Poncelet family, while that of the perimeter centroid $C_1$ is not, in general, a conic \cite{sergei2016-com}. However, as A. Akopyan reminded us \cite{akopyan2021-private}:

\begin{remark}
if a polygon circumscribes a circle (of center $O$), $C_1,C_2,O$ are collinear and  $(C_1-O)=(3/2)(C_2-O)$.
\end{remark}

\noindent Therefore, and analogously to \cite[Corollary 2]{garcia2022-steiner-soddy}:

\begin{corollary}
Over $\B'$, the locus of the perimeter centroid is a conic.
\end{corollary}

Let $\P'$ be the conic to which $\B'$ is inscribed.

\begin{conjecture}
Over the polar polygons of $\B'$ with respect to $\P'$, the locus of the perimeter centroid is never a conic.
\end{conjecture}

\section{A conserved quantity}
\label{sec:conserved}
As in \cref{app:bic}, let $\B$ denote a bicentric family of $N$-gons inscribed to a circle $\C=(O,R)$, and circumscribed about a second, nested circle $\C'$. Let $d_i$ denote the perpendicular distance from the bicentric circumcenter $O$ to side $P_i P_{i+1}$.

\begin{figure}
    \centering
    \includegraphics[trim=10 10 900 10,clip,width=.7\textwidth,frame]{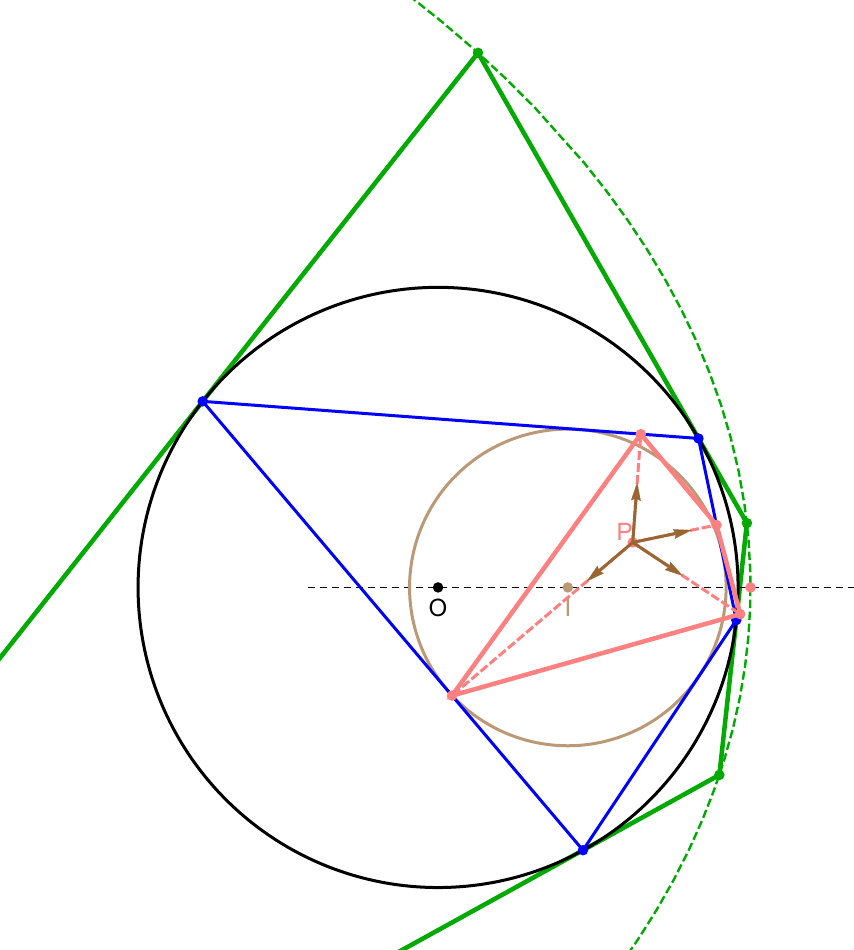}
    \caption{An $N=4$ bicentric polygon is shown (blue). Without loss of generality, in the case shown the circumcenter $O$ is interior to the incircle, i.e., the polar family (green) is ellipse-inscribed. Also shown is the pedal polygon (pink) with respect to a point $P$ in the interior of the circumcircle and the unit vectors (brown) along each perpendicular dropped from $P$ onto the sides. }
    \label{fig:bic-di}
\end{figure}




Referring to \cref{fig:bic-di}:


\begin{lemma}
Over $\B$, the quantity $\sum{d_i}$ is conserved.
\label{lem:bic-di}
\end{lemma}

The argument below was kindly provided by A. Akopyan \cite{akopyan2021-private}. 

\begin{proof}
The above statement is equivalent to stating that over $\B$ the sum of unit vectors from a point $P$ in the direction perpendicular to bicentric sides is constant. In turn, the latter is a corollary of the well-known fact that over $\B$, the centroid of the touchpoints of sidelines with $\C'$ is stationary.  
\end{proof}

\noindent Let $\theta_i$, $i=1,\ldots,N$, denote the angles interior to a polygon $\B$.

\begin{proposition}
For all $N$, the porism of polygons polar to $\B$ with respect to its circumcenter conserves $\sum_{i=1}^N{\sin{\theta_i/2}}=(1/R)\sum{d_i}$.
\end{proposition}

\begin{proof}
The vertices of the tangential polygon are the poles of each side of $\B$ with respect to the circumcircle. Therefore, said vertices are at a distance $D_i=R^2/d_i$ from the $O$. Since $\sin{\theta_i/2}=R/D_i=d_i/R$, per \cref{lem:bic-di}, the claim follows.
\end{proof}

Note that in general, $\theta_i$ is the directed angle $P_{i-1} P_i P_{i+1}$. In the case when $r<d$, the tangential polygon will be inscribed in two branches of a hyperbola. There are only two cases: Either (i) all vertices lie on a first proximal branch of the hyperbola, or (ii) all but one vertex $P_k$ will lie on said branch, with $P_k$ lying on the distal branch. In case (i), all $\theta_i$ are positive whereas in (ii) all are positive except for $\theta_k$. Furthermore, in this case, the supplement of angles $\theta_{k-1}$ and $\theta_{k+1}$ need to be used in the sum. So the invariant sum becomes
\begin{align*}
& \sin\frac{\theta_1}{2} +
\ldots +
\sin\frac{\pi-\theta_{k-1}}{2}
-\sin\frac{\theta_k}{2}
+\sin\frac{\pi-\theta_{k+1}}{2}+
\ldots+
\sin\frac{\theta_N}{2} = \\
& \sin\frac{\theta_1}{2} +
\ldots +
\cos\frac{\theta_{k-1}}{2}
-\sin\frac{\theta_k}{2}
+\cos\frac{\theta_{k+1}}{2}+
\ldots+
\sin\frac{\theta_N}{2}\;\;\cdot
\end{align*}

\section*{Acknowledgements}
\noindent We would like to thank A. Akopyan, B. Gibert, P. Moses, and A. Zaslavsky for their invaluable help with questions and proofs. The second author is fellow of CNPq and coordinator of Project PRONEX/ CNPq/ FAPEG 2017 10 26 7000 508.

\appendix
\section{Vertex parametrizations}
\label{app:vtx}
\subsection{Parabola-inscribed triangles}
 A 3-periodic orbit $P_i=[x_i,y_i]=[  -y_i^2/(4f),y_i]$ is such that
 \begin{align*}
 y_2&=  {\frac { 2\left(1- \sqrt {2}  \right)  \left( 4\,fy_1+\Delta\right) f}{8\,{f}^{2}\sqrt {2}-12\,{f}^{2}+y_1^{2}}},\;\;\;y_3=  \frac { 2\left(  \sqrt {2} -1\right) f\Delta}{8\,{f}^{2}\sqrt {2}-
12\,{f}^{2}+y_1^{2}},
\end{align*}
where $\Delta=\sqrt {16(  8\sqrt{2}-11)f^4 + 8f^2y_1^2 + y_1^4}$.

\subsection{Hyperbola-inscribed polar triangles}
 A 3-periodic orbit $Q_i=[q_{1,i},q_{2,i}]$ is such that
 {\small
 \begin{align*} Q_1&=\left( 1+\sqrt {2} \right) \left[{\frac { \left( 4\,fy_1+\Delta
 \right) y_1}{ 2(2\sqrt{2} + 3)y_1^2 - 8f^2}}, {\frac { \left(  1+\sqrt {2} \right) y_1^{3}-4\,
 \left( 1+\sqrt {2} \right) {f}^{2}y_1-2\,\Delta\,f}{2(2\sqrt{2} + 3)y_1^2 - 8f^2}}\right],\\
Q_2&=\left( 1+\sqrt {2} \right) \left[{\frac { \left( 4\,fy_1-\Delta
 \right) y_1}{2(2\sqrt{2} + 3)y_1^2 - 8f^2}},{\frac { \left( 1+\sqrt {2} \right) y_1^{3}- 4\left( 1+
\sqrt {2}  \right) {f}^{2}y_1+2\,\Delta\,f}{2(2\sqrt{2} + 3)y_1^2 - 8f^2}}\right],\\
Q_3&= \left(1+\sqrt{2}\right)\left[{\frac { \left( 5 - 3\sqrt{2} \right)  \left( ( 1+2\,
\sqrt {2})y_1^{2}-28\,{f}^{2}  \right) f}{7(( 3\, +2
\,\sqrt {2})y_1^{2}-4\,{f}^{2})}},-{\frac {8   {f}^{2}y_1\,}{(3\,+2\, \sqrt 
{2})y_1^{2}-4\,{f}^{2}}}\right],
\end{align*}
where $\Delta=\sqrt{y_1^4 + 8f^2 y_1^2 + 16(  8\sqrt{2}-11)f^4}$.
}

\subsection{Parabola-inscribed quadrilaterals}

 A 4-periodic orbit $P_i=[x_i,y_i]=[-\frac{1}{4f}   y_i^2,y_i]$ is such that:
 \begin{align*}
 y_2&=\frac { \left( 2\,\sqrt { \sqrt {5}-2}\;\Delta_1+4\,f y_1\, \left( 3-
\sqrt {5}  \right)  \right) f}{4\,{f}^{2}\sqrt {5}-8\,{f}^{2}-y_1^{2}},\;\;\;y_3=  \frac { 4\left( 2-\sqrt {5} \right) f^{2}}{y_1}, \\
y_4&= -{\frac { \left( 2\,\sqrt { \sqrt {5}-2}\;\Delta_1+4\,fy_1\, \left( 
\sqrt {5}-3 \right)  \right) f}{4\,{f}^{2}\sqrt {5}-8\,{f}^{2}-y_1^{2}}},
 \end{align*}
where $\Delta_1=\sqrt{y_1^{4}+8\,{f}^{2}y_1^{2}+16\, \left( 9-4\,\sqrt {5}
 \right) {f}^{4}}$.
 
\subsection{Hyperbola-inscribed polar quadrilaterals}

 A 4-periodic orbit $P_i=[p_i,q_i] $ is such that:
{\small
\begin{align*}
   p_1&={\frac {\sqrt { \sqrt {5}-2} \left( \Delta_1+6\,fy_1\,\sqrt {5}
\sqrt { \sqrt {5}-2}+14\,fy_1\,\sqrt { \sqrt {5}-2} \right)   y_1}{4\,y_1^{2}+2\,y_1^{2}\sqrt {5}-8\,f^{2}}}
\\ 
q_1&=
{\frac { \left( 2\,\sqrt {\sqrt {5}+2}\; \Delta_1\,f+4\,{f}^{2}y_1-{{
\it y1}}^{3} \right)  \left( 4\,{f}^{2}\sqrt {5}+8\,{f}^{2}+y_1
^{2} \right) }{32\,{f}^{4}-32\,{f}^{2}y_1^{2}-2\,y_1^{4}
}}\\
p_2&= \frac{2\,\sqrt { \sqrt {5}-2} \left( \Delta_1+ 2\sqrt{2}(  \sqrt{5}-1)  y_1 \right)  \left( 
 \left(  \sqrt {5} -2\right) y_1^{2}+4\,f^{2} \right) f^{2}}{y_1\, \left( 16\,f^{4}-16\,f^{2}y_1^{2}-y_1^{4}
 \right) 
 }
  \\
 q_2&= \sqrt {\sqrt {5}+2} \left( y_1\,
\Delta_1+2\,\sqrt {\sqrt {5}+2}fy_1^
{2}-8\, \left( \sqrt {5}-2 \right) ^{3/2}{f}^{3} \right)  \left(  (\sqrt {5}-2)y_1^{2}+4\,{f}^{2}  \right) f
\\
p_3&= \frac{-2\,\sqrt { \sqrt {5}-2} \left( \Delta_1-2\,\sqrt {2}\sqrt {\sqrt {5}-1}
fy_1\right)  \left(y_1^{2}\sqrt {5}+4\,{f}^{2}-2\,  
y_1^{2} \right) {f}^{2}
}{y_1\, \left( 16\,f^{4}-16\,f^{2}y_1^{2}-y_1^{4}
 \right)} \\
q_3&= \frac{-\sqrt {\sqrt {5}+2} \left( {\it y1}\,\Delta_1-2\,\sqrt {\sqrt {5}+2}f y_1^{2}+8\, \left(  \sqrt {5}-2 \right) ^{3/2}{f}^{3} \right) 
 \left(  (\sqrt {5}-2)y_1^2+4\,{f}^{2}  \right) f
}{y_1\, \left( 16\,f^{4}-16\,f^{2}y_1^{2}-y_1^{4}
 \right)} \\
p_4&=   \frac{ \sqrt { \sqrt {5}-2}   \left( \Delta_1-2\sqrt{2}(\sqrt{5}-1) f y_1 \right)  \left( (\sqrt {5}-2)y_1^{2}+4\,{f}^{2} \right) y_1
}{ y_1\, \left( 16\,f^{4}-16\,f^{2}y_1^{2}-y_1^{4}
 \right)}\\
q_4&=  \frac{ \left( -2\,f\Delta_1\,\sqrt { \sqrt {5}-2}+4\,{f}^{2}y_1-y_1^{3} \right)  \left( 4\,{f}^{2}\sqrt {5}+8\,{f}^{2}+y_1^{2}
 \right) 
}{y_1\, \left( 16\,f^{4}-16\,f^{2}y_1^{2}-y_1^{4}
 \right)} \\
\end{align*}
}

\section{Relation to the bicentric family}
\label{app:bic}
Referring to \cref{fig:bic345}, the bicentric family $\B$ of $N$-gons is a family of Poncelet $N$-gons inscribed in a circles $\C=(O,R_b)$, and circumscribed about another circle $\C'=(O',r_b)$. Let $d=|O-O'|$. Relations between $d,R,r_b$ are known for many ``low N'' and are listed in \cite[Poncelet's porism]{mw}. 

\begin{figure}
    \centering
    \includegraphics[trim=300 2 2 20,clip,width=\textwidth,frame]{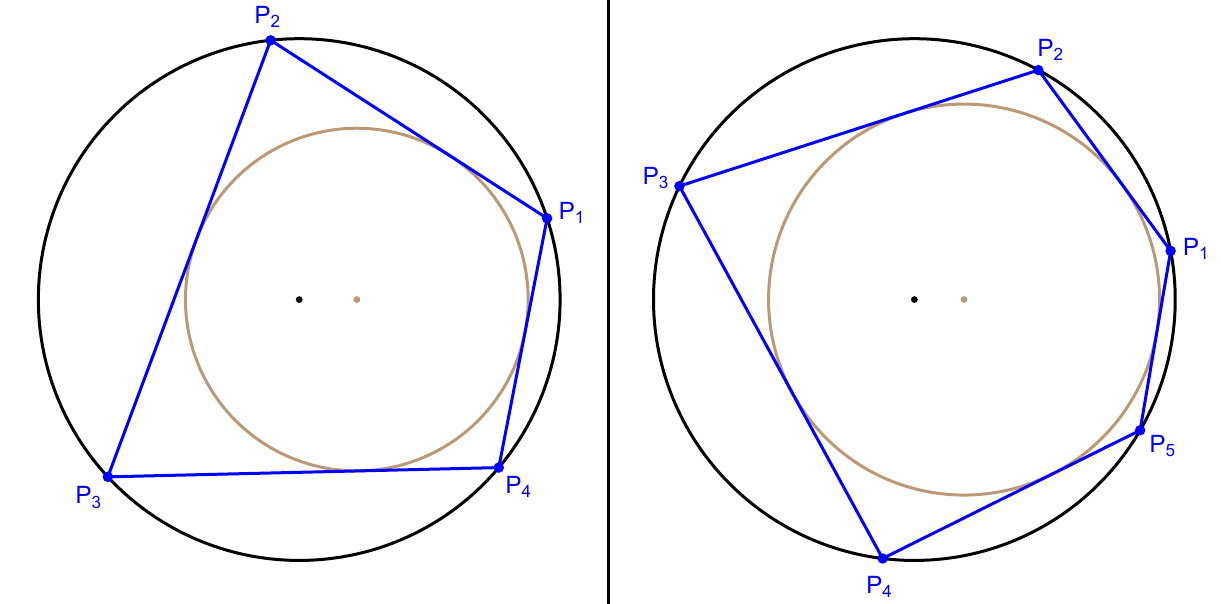}
    \caption{The bicentric family is a family of Poncelet polygons interscribed between two circles. Shown are the $N=4$ (left) and $N=5$ (right) convex cases.}
    \label{fig:bic345}
\end{figure}

\begin{definition}[Polar polygon]
Given a polygon $P$, its {\em polar polygon} $P'$ with respect to a conic $\C$ is bounded by the tangents to $\C$ at the vertices of $\P$.
\end{definition}

\begin{proposition}
The polar family $\B'$ of $\B$ with respect to $\C$ is an ellipse, parabola, or hyperbola-inscribed if $d$ is smaller, equal, or greater than $R'$, respectively ($O$ is interior, on the boundary, or exterior to $C'$, respectively). Furthermore, one of the foci coincides with $O'$.
\end{proposition} 

As shown in \cref{fig:bic-tang-n4}, when the polar family is hyperbola-inscribed, there are two layouts for its vertices: either (i) all lie on the branch of the hyperbola closest to the incenter of the family, or (ii) all but one lie on said branch, while the remaining one lies on the ``other'' branch. 

\begin{figure}
    \centering
    \includegraphics[trim=450 50 5 20,clip,width=.7\textwidth,frame]{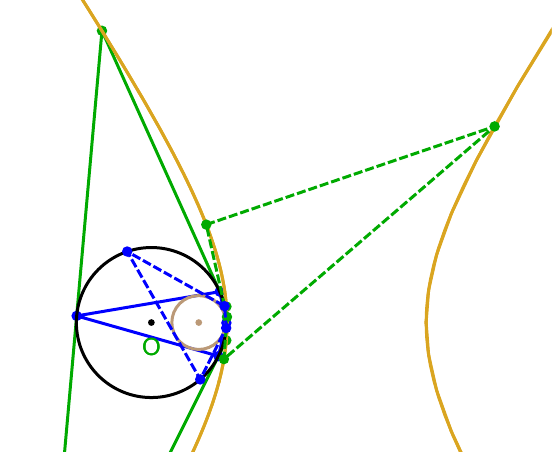}
    \caption{If the circumcenter $O$ is exterior to the incircle of a bicentric polygon (blue), the polar (i.e., tangential) family will be hyperbola (gold) inscribed. Over the family there are two configurations: (i) solid green: all vertices lie on one branch of the hyperbola; (ii) dashed green: all but one vertex lie on the branch proximal to the incenter, while a lone one lies on the opposite branch.}
    \label{fig:bic-tang-n4}
\end{figure}

\begin{proposition}
The parabola $\P$ which is the polar image of $\B$ with $d=r_b$, has focal distance $f=R_b^2/(2r_b)$.
\end{proposition}

\begin{proof}
Let $O=(0,0)$. Consider a polygon in $\B$ with a vertical side $P_1 P_2$ tangent to the incircle at $(2r_b,0)$. The vertex $V$ of $\P$ is the pole of said side which can be obtained as the inversion of point $(2r_b,0)$ with respect to the circumcircle. This yields the result.
\end{proof}

\bibliographystyle{maa}
\bibliography{refs,refs_rgk,refs_rgk_private,refs_rgk_media}

\begin{thebibliography}{10}
\expandafter\ifx\csname urlstyle\endcsname\relax
 \providecommand{\url}[1]{doi:\discretionary{}{}{}#1}\else
 \providecommand{\url}{doi:\discretionary{}{}{}\begingroup
  \urlstyle{rm}\Url}\fi

\bibitem{akopyan2021-private}
Akopyan, A. (2021).
\newblock Private communication.
\newblock Private Communication.

\bibitem{akopyan2020-invariants}
Akopyan, A., Schwartz, R., Tabachnikov, S. (2020).
\newblock Billiards in ellipses revisited.
\newblock \emph{Eur. J. Math.}
\newblock {doi}:10.1007/s40879-020-00426-9.

\bibitem{akopyan2007-conics}
Akopyan, A.~V., Zaslavsky, A.~A. (2007).
\newblock \emph{Geometry of Conics}.
\newblock Providence, RI: Amer. Math. Soc.

\bibitem{bialy2020-invariants}
Bialy, M., Tabachnikov, S. (2020).
\newblock {Dan Reznik's} identities and more.
\newblock \emph{Eur. J. Math.}
\newblock {d}oi:10.1007/s40879-020-00428-7.

\bibitem{bos-1987}
Bos, H. J.~M., Kers, C., Raven, D.~W. (1987).
\newblock Poncelet's closure theorem.
\newblock \emph{Expo. Math.}, 5: 289--364.

\bibitem{caliz2020-area-product}
Chavez-Caliz, A. (2020).
\newblock More about areas and centers of {Poncelet} polygons.
\newblock \emph{Arnold Math J.}
\newblock {d}oi:10.1007/s40598-020-00154-8.

\bibitem{centina2016a}
Del~Centina, A. (2016).
\newblock Poncelet's porism: a long story of renewed discoveries, {I}.
\newblock \emph{Arch. Hist. Exact Sci.}, 70(1): 1--122.

\bibitem{dragovic11}
Dragovi\'{c}, V., Radnovi\'{c}, M. (2011).
\newblock \emph{{P}oncelet Porisms and Beyond: Integrable Billiards,
  Hyperelliptic Jacobians and Pencils of Quadrics}.
\newblock Frontiers in Mathematics. Basel: Springer.

\bibitem{corentin2021-circum}
Fierobe, C. (2021).
\newblock On the circumcenters of triangular orbits in elliptic billiard.
\newblock \emph{J. of Dyn. and Control Sys.}, 27: 693--705.

\bibitem{garcia2022-steiner-soddy}
Garcia, R., Gheorghe, L., Reznik, D. (2022).
\newblock Exploring the {S}teiner-{S}oddy porism.
\newblock {arXiv}:2201.0222.

\bibitem{garcia2021-impa}
Garcia, R., Reznik, D. (2021).
\newblock \emph{Euclidean Properties of Poncelet Triangles: Experimental
  Discovery and Analysis}.
\newblock Rio de Janeiro: IMPA.
\newblock {isbn:}978-65-89124-43-6.

\bibitem{garcia2020-family-ties}
Garcia, R., Reznik, D. (2021).
\newblock Family ties: Relating {P}oncelet 3-periodics by their properties.
\newblock \emph{J. Croatian Soc. for Geom. \& Gr. (KoG)}, 25(25): 3--18.

\bibitem{garcia2020-ellipses}
Garcia, R., Reznik, D., Koiller, J. (2020).
\newblock Loci of 3-periodics in an elliptic billiard: why so many ellipses?
\newblock arXiv:2001.08041.

\bibitem{garcia2020-new-properties}
Garcia, R., Reznik, D., Koiller, J. (2021).
\newblock New properties of triangular orbits in elliptic billiards.
\newblock \emph{Am. Math. Monthly}, 128(10): 898--910.

\bibitem{helman2021-power-loci}
Helman, M., Laurain, D., Garcia, R., Reznik, D. (2021).
\newblock Invariant center power and loci of {P}oncelet triangles.
\newblock \emph{J. Dyn. \& Contr. Sys.}
\newblock {doi}:10.1007/s10883-021-09580-z.

\bibitem{etc}
Kimberling, C. (2019).
\newblock Encyclopedia of triangle centers.
\newblock \url{faculty.evansville.edu/ck6/encyclopedia/ETC.html}.

\bibitem{maple2019}
Maplesoft (2019).
\newblock Maple.
\newblock A division of {W}aterloo {M}aple Inc.

\bibitem{odehnal2011-poristic}
Odehnal, B. (2011).
\newblock Poristic loci of triangle centers.
\newblock \emph{J. Geom. Graph.}, 15(1): 45--67.

\bibitem{pamfilos2004}
Pamfilos, P. (2004).
\newblock On some actions of {$D_3$} on a triangle.
\newblock \emph{Forum Geometricorum}, 4: 157--176.

\bibitem{reznik2020-ballet}
Reznik, D., Garcia, R., Koiller, J. (2020).
\newblock The ballet of triangle centers on the elliptic billiard.
\newblock \emph{Journal for Geometry and Graphics}, 24(1): 079--101.

\bibitem{olga14}
Romaskevich, O. (2014).
\newblock On the incenters of triangular orbits on elliptic billiards.
\newblock \emph{Enseign. Math.}, 60: 247--255.

\bibitem{sergei2016-com}
Schwartz, R., Tabachnikov, S. (2016).
\newblock Centers of mass of {P}oncelet polygons, 200 years after.
\newblock \emph{Math. Intelligencer}, 38(2): 29--34.

\bibitem{schwartz2016-com}
Schwartz, R., Tabachnikov, S. (2016).
\newblock Centers of mass of {P}oncelet polygons, 200 years after.
\newblock \emph{Math. Intelligencer}, 38(2): 29--34.

\bibitem{skutin2013-isogonal}
Skutin, A. (2013).
\newblock On rotation of an isogonal point.
\newblock \emph{J. Classical Geom.}, 2: 66–67.

\bibitem{mw}
Weisstein, E.~W. (2002).
\newblock \emph{{CRC concise encyclopedia of mathematics (2nd ed.)}}.
\newblock Boca Raton, FL: Chapman and Hall/CRC.

\bibitem{mathematica_v10}
Wolfram, S. (2019).
\newblock Mathematica, version 10.0.

\bibitem{zaslasvsky2021-private}
Zaslavsky, A. (2021).
\newblock Private communication.

\bibitem{zaslavksy2003}
Zaslavsky, A., Kosov, D., Muzafarov, M. (2003).
\newblock Trajectories of remarkable points of the poncelet triangle (in
  russian).
\newblock \emph{Kvant}, 2: 22–25.

\end{thebibliography}

\end{document}